\documentclass[11pt]{amsart}
\usepackage[margin=2.5cm]{geometry}
\usepackage{enumerate}
\usepackage{amsmath}
\usepackage{amssymb,latexsym}
\usepackage{amsthm}
\usepackage{color}
\usepackage[dvipsnames]{xcolor}
\usepackage{cancel}
\usepackage{listings}
\usepackage{graphicx}
\usepackage{cite}
\usepackage{changes}
\usepackage{mathdots}
\usepackage{float}

\makeatletter
\renewcommand{\tocsection}[3]{%
  \indentlabel{\@ifnotempty{#2}{\bfseries\ignorespaces#1 #2\quad}}\bfseries#3}
\renewcommand{\tocsubsection}[3]{%
  \indentlabel{\@ifnotempty{#2}{\ignorespaces#1 #2\quad}}#3}

\renewcommand{\l@paragraph}[4]{%
  \@tocline{4}{0pt}{1.5em}%
  {\indentlabel{\@ifnotempty{#1}{\ignorespaces#1\quad}}}%
  {}%
}
\makeatother

\usepackage{etoolbox}
\makeatletter
\patchcmd{\@setaddresses}{\indent}{\noindent}{}{}
\patchcmd{\@setaddresses}{\indent}{\noindent}{}{}
\patchcmd{\@setaddresses}{\indent}{\noindent}{}{}
\patchcmd{\@setaddresses}{\indent}{\noindent}{}{}
\makeatother

\usepackage[linktocpage]{hyperref}
\hypersetup{
  colorlinks   = true, 
  urlcolor     = blue, 
  linkcolor    = Purple, 
  citecolor   = red 
}
\usepackage{comment}
\usepackage{amscd}
\usepackage{mathtools}
\usepackage{soul}

\DeclareMathOperator{\C}{\mathcal{C}}

\DeclareMathOperator{\GL}{GL}

\DeclareMathOperator{\Sym}{Sym}
\usepackage{cleveref}

\theoremstyle{definition}
\newtheorem{theorem}{Theorem}[section]

\newtheorem{lemma}[theorem]{Lemma}
\newtheorem{corollary}[theorem]{Corollary}
\newtheorem{definition}[theorem]{Definition}
\newtheorem{proposition}[theorem]{Proposition}

\newtheorem{example}[theorem]{Example}

\newtheorem{remark}[theorem]{Remark}

\newcommand{\mc}[1]{\mathcal{#1}}

\newcommand{\F}{{\mathbb F}}

\newcommand{\Z}{{\mathbb Z}}

\newcommand{\w}{{\mathbf w}}

\newcommand{\fq}{{\mathbb F}_{q}}

\newcommand{\Fm}{{\mathbb F}_{q^m}}

\newcommand{\la}{\langle}
\newcommand{\ra}{\rangle}

\newcommand{\N}{\mathrm{N}}

\newcommand{\fqm}{\mathbb{F}_{q^m}}

\DeclareMathOperator{\im}{Im}

\DeclareMathOperator{\stab}{stab}

\DeclareMathOperator{\cyclic}{C}

\usepackage{comment}

\title{On the number of inequivalent linearized Reed-Solomon codes}
\date{}

\author[J. Mannaert]{Jonathan Mannaert}
\address{Jonathan Mannaert, \textnormal{Department of Mathematics and Data Science, Vrije Universiteit Brussel, Pleinlaan 2, 1050 Brussel, Belgium}}
\email{jonathan.mannaert@vub.be}

\author[M. Messia]{Marta Messia}
\address{Marta Messia, \textnormal{Department of Mathematics and Data Science, Vrije Universiteit Brussel, Pleinlaan 2, 1050 Brussel, Belgium}}
\email{marta.messia@vub.be}

\author[F. Zullo]{Ferdinando Zullo}
\address{Ferdinando Zullo, \textnormal{Dipartimento di Matematica e Fisica, Universit\`a degli Studi della Campania ``Luigi Vanvitelli'', Viale Lincoln, 5, I--\,81100 Caserta, Italy}}
\email{ferdinando.zullo@unicampania.it}

\begin{document}

\begin{abstract}
Linearized Reed-Solomon (LRS) codes form an important family of maximum sum-rank distance (MSRD) codes that generalize both Reed--Solomon codes and Gabidulin codes. 
In this paper we study the equivalence problem for LRS codes and determine the number of inequivalent codes within this family. 
Using the correspondence between sum-rank metric codes and systems of $\F_q$-subspaces, we analyze the stabilizer of the Gabidulin system and derive a characterization of equivalence between LRS codes. 
In particular, we prove that two LRS codes are equivalent if and only if the sets of norms that define the codes coincide up to multiplication by an element of $\F_q^\ast$. 
This description allows us to reduce the classification problem to the action of $\F_q^\ast$ on subsets of $\F_q^\ast$. 
As a consequence, we derive formulas for the number of inequivalent linearized Reed-Solomon codes and illustrate the results with explicit examples.
\end{abstract}

\maketitle

\tableofcontents

\section{Introduction}

Error-correcting codes endowed with the rank metric have attracted considerable attention in the last decades due to their rich algebraic structure and relevance in several applications. 
In the rank metric, the distance between two vectors over a finite field extension is defined as the rank of their difference when viewed over the base field. 
Rank-metric codes were introduced independently by Delsarte \cite{de78}, Gabidulin \cite{ga85a}, and Roth \cite{roth1991maximum}, and have since found applications in areas such as network coding, distributed storage, cryptography, and space–time coding; \cite{bartz2022rank}.

A central role in the theory of rank-metric codes is played by Maximum Rank Distance (MRD) codes, which attain the Singleton-like bound for the rank metric. 
Among these, Gabidulin codes form one of the most important and well-understood families. 
These codes can be described as evaluation codes of linearized polynomials over $\F_{q^m}$ evaluated at elements that are linearly independent over $\F_q$. 
Their strong algebraic structure makes them a natural analogue of Reed–Solomon codes in the rank-metric setting.

More recently, the sum-rank metric has emerged as a natural generalization of both the Hamming and the rank metric, \cite{byrne2021fundamental,martinez2022codes,martinez2019universal}.  
In this metric, vectors are partitioned into blocks and the weight is defined as the sum of the ranks of the blocks. 
Sum-rank metric codes unify several important coding-theoretic models and have proven particularly useful in multishot network coding and convolutional coding. 
Analogously to the rank-metric case, the maximum Sum-Rank Distance (MSRD) codes are those that achieve the corresponding Singleton-type bound.

A prominent family of MSRD codes is given by the \emph{linearized Reed--Solomon (LRS) codes}, introduced by Martínez-Peñas \cite{Martinez2018skew}. 
These codes can be viewed as a natural generalization of both Reed–Solomon codes and Gabidulin codes. 
Indeed, Reed–Solomon codes arise when each block has length one, while Gabidulin codes correspond to the case of a single block. 
From the algebraic perspective, LRS codes can be constructed as evaluation codes of linearized polynomials in a skew-polynomial framework, and they inherit many structural properties from both classical families.

Despite their importance, several structural aspects of linearized Reed–Solomon codes remain poorly understood. 
In particular, the problem of determining when two such codes are equivalent under the natural group of sum-rank isometries has received limited attention. 
Understanding equivalence classes is a fundamental problem in coding theory, as it allows one to classify codes up to the natural symmetries of the metric and to determine the number of essentially different codes with given parameters.

In this paper, we study the equivalence problem for linearized Reed–Solomon codes and determine the number of inequivalent codes within this family. 
Our approach is geometric and relies on the correspondence between sum-rank metric codes and systems of $\F_q$-subspaces. 
Using this perspective, we analyze the stabilizer of the Gabidulin system and characterize when two linearized Reed–Solomon systems are equivalent.

More precisely, we prove that two linearized Reed–Solomon codes are equivalent if and only if the sets of norms associated with their defining parameters coincide up to multiplication by an element of $\F_q^\ast$. 
This characterization allows us to reduce the equivalence problem to the study of the action of the multiplicative group $\F_q^\ast$ on subsets of $\F_q^\ast$. 
Using this description, we derive formulas for the number of inequivalent linearized Reed–Solomon codes and provide explicit counting results.
These results extend the work of Schmidt and Zhou \cite{schmidt2018number}, who determined the number of inequivalent Gabidulin codes in the square case.  The family of linearized Reed–Solomon codes contains many inequivalent codes. Understanding and counting these inequivalent classes is therefore a natural problem.

The paper is organized as follows. 
In Section~\ref{sec:2} we review the necessary background on rank-metric codes, linearized polynomials, and Gabidulin codes. 
In Section~\ref{sec:3} we introduce sum-rank metric codes and linearized Reed–Solomon codes. 
Section~\ref{sec:4} is devoted to the study of the stabilizer of the Gabidulin system. 
In Section~\ref{sec:5} we characterize equivalence classes of linearized Reed–Solomon codes. 
In Section~\ref{sec:6} we derive formulas for counting inequivalent linearized Reed–Solomon codes and present explicit examples.
Finally, in Section~\ref{sec:7} we discuss about possible open problems.

\section{Rank-metric codes and Gabidulin codes}\label{sec:2}
In this section, we introduce the notation for rank-metric codes. These codes can be introduced in a variety of perspectives. First we consider them from the traditional point of view, thereafter we shift towards the vector, and other geometric frameworks. Here we explicitly focus on the linearized polynomial framework, $q$-systems and Gabidulin codes, as they will be useful later on.

Rank-metric codes are error-correcting codes in which the distance between codewords is measured using the rank metric, defined as the rank of the difference between two matrices over a finite field. Introduced in the work of 
Delsarte \cite{de78}, Gabidulin \cite{ga85a}, and Roth \cite{roth1991maximum}, these codes have attracted significant attention due to their strong algebraic structure and their ability to correct errors that affect data in a correlated or structured way. Rank-metric codes arise naturally in scenarios where information is represented in matrix form, such as network coding, distributed storage, and space-time coding; see \cite{bartz2022rank,gorla2021rank}.
Here, we will describe them in the vector framework.

\subsection{The vector framework}
For every $m,n\in{\mathbb{N}}$, denote by $\mathbb{F}_{q^m}^n$ the vector space over $\mathbb{F}_{q^m}$ of dimension $n$. Here we introduce the rank weight, let $x=(x_1,\ldots,x_n)\in\mathbb{F}_{q^m}^n$ we define its rank weight as follows\begin{equation*}
 \w_R(x)=\dim_{\mathbb{F}_q}(\langle x_1,\ldots,x_n\rangle _{\mathbb{F}_q})   
\end{equation*} 
and the rank distance of $x,y\in\mathbb{F}_{q^m}^n$ as $d_R(x,y)=\w_R(x-y)$. 

From this we can introduce the following definition.

\begin{definition}
A \textbf{rank-metric code} is a $k$-dimensional $\mathbb{F}_{q^m}$-subspace of $\mathbb{F}_{q^m}^n$, and we will write the parameters of such a code by $[n,k]_{q^m/q}$ or by $[n,k,d]_{q^m/q}$ if $d$ is its minimum distance, that is
\begin{equation*}
    d=\min\{d_R(x,y)\,:\,x,y\in\mc{C},x\neq y\}.
\end{equation*}
\end{definition}

The linear isometries of $(\mathbb{F}_{q^m}^n,d_R)$ are fully characterized (see e.g. \cite{berger2003isometries}).

\begin{theorem}
    Let $\phi:\mathbb{F}_{q^m}^n\rightarrow\mathbb{F}_{q^m}^n$ be an $\mathbb{F}_{q^m}$-linear isometry of $(\mathbb{F}_{q^m}^n,d_R)$, then there exist $\alpha\in\mathbb{F}_{q^m}^*$ and $A\in\GL(n,q)$ such that $\phi(x)=\alpha xA$ for every $x\in\mathbb{F}_{q^m}^n$.
\end{theorem}

As a consequence, we give the definition of equivalent codes.

\begin{definition}
    Let $\mc{C},\mc{D}\subseteq\mathbb{F}_{q^m}^n$ be two $\mathbb{F}_{q^m}$-linear and $k$-dimensional codes. The codes $\mc{C},\mc{D}$ are said to be \textbf{equivalent} if there exists an $\mathbb{F}_{q^m}$-linear isometry $\phi$ of $(\mathbb{F}_{q^m}^n,d_R)$ such that $\phi(\mc{C})=\mc{D}$.
\end{definition}

Most of the codes we will consider are \textbf{non-degenerate}, i.e. those for which the columns of any generator matrix of $\C$ are $\fq$-linearly independent; see \cite{alfarano2022linear}.

Similarly as for the Hamming metric, one can prove a Singleton-like bound for rank-metric codes. 

\begin{theorem}(see \cite{de78}) \label{th:singletonrank}
    Let $\C$ be an $[n,k,d]_{q^m/q}$ code.
Then 
\begin{equation}\label{eq:boundgen}
mk \leq \max\{m,n\}(\min\{m,n\}-d+1).\end{equation}
\end{theorem}

An $[n,k,d]_{q^m/q}$ code is called \textbf{Maximum Rank Distance code} (or in a short form \textbf{MRD code}) if its parameters reach the bound \eqref{eq:boundgen}.

\subsection{Linearized polynomials}

Let $s\in\{1,\ldots,m\}$ be such that $(s,m)=1$, a \textbf{linearized polynomial} (or \textbf{$q^s$-polynomial}) is a polynomial of the form 
\begin{equation*}  
L(x)=\sum_{i=0}^t\alpha_ix^{q^{si}}\in\mathbb{F}_{q^m}[x].
\end{equation*}
Any linearized polynomial $L(x)$, can be considered as a linear map over $\mathbb{F}_{q^m}$. Such interpretation results into the introduction of $\ker(L)$ and and $\im(L)$, which are both $\mathbb{F}_q$-subspaces. Here we have $\dim_{\mathbb{F}_q}(\ker(L))+\dim_{\mathbb{F}_q}(\im(L))=m$. 
In this context, we can introduce the notion of $q^s$-\textbf{degree} of a linearized polynomial 
\begin{equation*}
    \deg_{q^s}(L)=\max\{i\in\{1,\ldots,m\}\,:\,x^{q^{si}} \text{ is a monomial of }L\}.
\end{equation*}
As classically for polynomials, the degree bounds the number of roots of the polynomial in the following way (see e.g. \cite[Theorem~5]{Galois_extensions_and_subspaces_of_alternating_bilinear_forms_with_special_rank_properties}).

\begin{theorem}\label{thm:boundnumbroots}
   Consider $s \in\{1,\ldots,m\}$ with $\gcd(s,m)=1$.
   Let $L(x)=\sum_{i=0}^t\alpha_ix^{q^{si}}\in\mathbb{F}_{q^m}[x]$ be a nonzero $q^s$-polynomial and let $t$ be its $q^s$-degree. Then $\dim_{\fq}(\ker(L))\leq t$.
\end{theorem}

Denote by $\mc{L}_{m,q^s}$ the set of $q^s$-polynomials over $\mathbb{F}_{q^m}$ and note that $(\mc{L}_{m,q^s},+,\circ,\cdot)$ is an $\mathbb{F}_q$-algebra, where $\circ$ is defined between two linearized polynomials as the composition of the associated functions.
Since $\{(x^{q^{sm}}-x)\circ f(x)\,:\,f\in\mc{L}_{m,q^s}\}$ is a two-sided ideal of $\mc{L}_{m,q^s}$, then the quotient
\begin{equation*}
\tilde{\mc{L}}_{m,q^s}=\mc{L}_{m,q^s}/(x^{q^{sm}}-x)
\end{equation*}
is again an $\mathbb{F}_q$-algebra which turns out to be isomorphic to the $\mathbb{F}_q$-algebra of the $\fq$-linear endomorphisms of $\fqm$.

Let $L(x)=a_0x+\ldots+a_{m-1}x^{q^{s(m-1)}}\in\tilde{\mc{L}}_{m,q^s}$, the matrix
\begin{equation*}
    D_{L}:=\begin{pmatrix}
        a_0&a_1&\ldots&a_{m-1}\\
        a_{m-1}^{q^s}&a_0^{q^s}&\ldots&a_{m-2}^{q^s}\\
        \vdots&&\ldots&\vdots\\
        a_1^{q^{s(m-1)}}&a_2^{q^{s(m-1)}}&\ldots&a_0^{q^{s(m-1)}}
    \end{pmatrix}
\end{equation*}
is called the \textbf{Dickson matrix} associated with $L$. 
An important reason Dickson matrices are widely used in the literature is $\operatorname{rk}(D_L) = \dim_{\mathbb{F}_q}(\mathrm{Im}(L))$. In particular, this implies that the polynomial $L$ is invertible if and only if $D_L$ is invertible.

For more details on linearized polynomials, we refer to \cite{lidl1997finite,wu2013linearized}.

Gabidulin codes can be introduced naturally in this framework as evaluation codes of linearized polynomials. 
More precisely, they are obtained by evaluating $q^s$-linearized polynomials of bounded $q^s$-degree at elements of $\F_{q^m}$ that are linearly independent over $\F_q$.

\begin{definition}
Let $\boldsymbol{\alpha}=(\alpha_1,\ldots,\alpha_n)$ such that $\alpha_1,\ldots,\alpha_n\in\F_{q^m}$ are $\F_q$-linearly independent. Define
\[
G_{k,s,m}[\boldsymbol{\alpha}]
:=
\{(f(\alpha_1),\ldots,f(\alpha_n)) : f\in\mc{L}_{m,q^s} \text{ s.t. } \deg_{q^s}(f)\leq k-1\}
\subseteq \F_{q^m}^n .
\]
Then $G_{k,s,m}[\boldsymbol{\alpha}]$ is a \textbf{Gabidulin code}.  
\end{definition}

Gabidulin codes are examples of MRD codes.
 
\begin{theorem}\cite{de78}
Let $\boldsymbol{\alpha}=(\alpha_1,\ldots,\alpha_n)$ such that $\alpha_1,\ldots,\alpha_n\in\F_{q^m}$ are $\F_q$-linearly independent and consider $G_{k,s,m}[\boldsymbol{\alpha}]=\{(f(\alpha_1),\ldots,f(\alpha_n)) : f\in\mc{L}_{m,q^s} \text{ s.t. } \deg_{q^s}(f)\leq k-1\}$ with $k\leq m$. Then $G_{k,s,m}[\boldsymbol{\alpha}]$ is an MRD code on $\F_{q^m}^n$.
\end{theorem}

\subsection{The geometric framework}
 
A key point in the theory of rank-metric codes has been the geometric viewpoint via systems. Indeed, this points out a connection between rank-metric codes and the combinatorics of $\fq$-subspaces in an $\fqm$-vector space. We recall this connection.

\begin{theorem}(see \cite{Randrianarisoa2020ageometric}) \label{th:connection}
Let $\C$ be a non-degenerate $[n,k,d]_{q^m/q}$ code and let $G$ be a generator matrix.
Let $U \subseteq \F_{q^m}^k$ be the $\F_q$-span of the columns of $G$.
The rank weight of an element $x G \in \C$, with $x \in \F_{q^m}^k$ is
\begin{equation}\label{eq:relweight}
\w_R(x G) = n - \dim_{\mathbb{F}_q}(U \cap x^{\perp}),\end{equation}
where $x^{\perp}=\{y \in \F_{q^m}^k \colon x \cdot y=0\}$ and $x \cdot y$ denotes the standard scalar product between $x$ and $y$. 
In particular,
\begin{equation} \label{eq:distancedesign}
d=n - \max\left\{ \dim_{\fq}(U \cap H)  \colon H\mbox{ is an } \F_{q^m}\mbox{-hyperplane of }\F_{q^m}^k  \right\}.
\end{equation}
\end{theorem}

The above theorem allows us to consider the notion of $q$-system.

\begin{definition}
    Let $\mc{U}$ be an $\mathbb{F}_q$-subspace of $\mathbb{F}_{q^m}^k$ such that $\langle \mc{U}\rangle _{\mathbb{F}_{q^m}}=\mathbb{F}_{q^m}^k$. We say that $\mc{U}$ is an $[n,k,d]_{q^m/q}$-\textbf{system} if $n=\dim_{\mathbb{F}_q}(\mc{U})$ and $d=n-\max\{\dim_{\mathbb{F}_q}(\mc{U}\cap H)\,:\, H \text{ is an } \mathbb{F}_q\text{-hyperplane of } \mathbb{F}_{q^m}^k\}$. 

    When not considering parameters explicitly, we will refer to $\mc{U}$ as a $q$-system.
    Moreover, we say that two $[n,k,d]_{q^m/q}$-systems $\mc{U}$ and $\mc{V}$ are \textbf{equivalent} if there exists an $\F_{q^m}$-linear isomorphism $\phi:\mathbb{F}_{q^m}^k\rightarrow\mathbb{F}_{q^m}^k$ such that $\phi(\mc{U})=\mc{V}$. 
\end{definition}

\begin{remark}
Directly from the previous definition we have that two $[n,k,d]_{q^m/q}$-systems $\mc{U}$ and $\mc{V}$ are equivalent if there exists $A\in\GL(k,q^m)$ such that $\mc{V}=\mc{U}A$.
\end{remark}

In \cite{Randrianarisoa2020ageometric}, see also \cite{alfarano2022linear}, the author proved that equivalent codes correspond to equivalent $q$-systems and conversely. This yields the following result.

\begin{theorem}
\label{th:corresp}
 Let $\mathfrak{C}[n,k,d]_{q^m/q}$ be the set of equivalence classes of non-degenerate $[n,k,d]_{q^m/q}$ rank-metric codes and $\mathfrak{U}[n,k,d]_{q^m/q}$ the set of equivalence classes of $[n,k,d]_{q^m/q}$-systems. There exists a one-to-one correspondence between $\mathfrak{C}[n,k,d]_{q^m/q}$ and $\mathfrak{U}[n,k,d]_{q^m/q}$.   
\end{theorem}

Thanks to Theorem~\ref{th:corresp}, rank-metric codes and $q$-systems can be studied interchangeably. In particular, every rank-nondegenerate $[n,k,d]_{q^m/q}$-code uniquely determines a $[n,k,d]_{q^m/q}$-system up to equivalence, and vice versa. Therefore, in the following we will freely move between these two perspectives, choosing the one that is more convenient for the problem at hand.

Thanks to the above correspondence, although Gabidulin codes were originally introduced as evaluation codes, they can be naturally described in terms of $q$-systems.

\begin{example}
Let $\boldsymbol{\alpha}=(\alpha_1,\ldots,\alpha_n)$ be such that $\alpha_1,\ldots,\alpha_n\in\F_{q^m}$ are $\F_q$-linearly independent and consider $G_{k,s,m}[\boldsymbol{\alpha}]$ a Gabidulin code. Then, the generator matrix of $G_{k,s,m}[\boldsymbol{\alpha}]$ is the Moore matrix
\[
G=
\begin{pmatrix}
\alpha_1 & \ldots & \alpha_n\\
\alpha_1^{q^s} & \ldots & \alpha_n^{q^s}\\
\vdots & & \vdots\\
\alpha_1^{q^{s(k-1)}} & \ldots & \alpha_n^{q^{s(k-1)}}
\end{pmatrix}
\in\F_{q^m}^{k\times n}.
\]

Using the correspondence described in Theorem~\ref{th:corresp}, we can determine the $q$-system associated with $G_{k,s,m}[\boldsymbol{\alpha}]$. Indeed, from the columns of the generator matrix we obtain
\begin{align*}
\mc{U}
&=
\langle (\alpha_1,\alpha_1^{q^s},\ldots,\alpha_1^{q^{s(k-1)}}),\ldots,
(\alpha_n,\alpha_n^{q^s},\ldots,\alpha_n^{q^{s(k-1)}})\rangle_{\F_q} \\
&=
\{(x,x^{q^s},\ldots,x^{q^{s(k-1)}})\;:\;
x\in\langle\alpha_1,\ldots,\alpha_n\rangle_{\F_q}\}.
\end{align*}
\end{example}

\begin{remark}
Let $s\in\{1,\ldots,m\}$ with $(s,m)=1$ and let $S\subseteq\F_{q^m}$ be an $\F_q$-subspace with $\dim_{\F_q}(S)=n$. Then there exist $\F_q$-linearly independent $\alpha_1,\ldots,\alpha_n\in\F_{q^m}$ such that $S=\langle\alpha_1,\ldots,\alpha_n\rangle_{\F_q}$. In this case
\[
G_{k,s,m}[S]
:=
\{(x,x^{q^s},\ldots,x^{q^{s(k-1)}}): x\in S\}
\]
is the $q$-system associated with the Gabidulin code $G_{k,s,m}[\boldsymbol{\alpha}]$.

When $S=\F_{q^m}$, we simply write $G_{k,s,m}:=G_{k,s,m}[\F_{q^m}]$.
We will refer to these $q$-systems as \textbf{Gabidulin systems}.
\end{remark}

For more details, we refer to \cite{sheekeysurvey,neri2026rank}.

\section{Sum-rank metric codes and linearized Reed-Solomon codes}\label{sec:3}

\subsection{Sum-rank metric codes}

Let $m,n_1,\ldots,n_t\in\mathbb{N}$ and denote by $\F_{q^m}^{n_i}$ the vector space over $\F_{q^m}$ of dimension $n_i$. Moreover, we will write $\mathbf{n}=(n_1,\ldots,n_t)$, $N=n_1+\ldots+n_t$ and
\begin{equation*}
    \F_{q^m}^{\mathbf{n}}=\bigoplus_{i=1}^t\F_{q^m}^{n_i}.
\end{equation*}
For a tuple of vectors $x=(x_1,\ldots,x_t)\in\F_{q^m}^{\mathbf{n}}$ we can define the sum-rank weight of $x$ as
\begin{equation*}
    \w_{SR}(x)=\sum_{i=1}^t\w_R(x_i)
\end{equation*}
and the sum-rank distance can be defined, for $x,y\in\F_{q^m}^{\mathbf{n}}$, as $d_{SR}(x,y)=\w_{SR}(x-y)$. With this notation of distance we can introduce the following definition.
\begin{definition}
    Let $k$ be a positive integer with $1 \le k \le N$. An $[\mathbf{n},k,d]_{q^m/q}$ \textbf{sum-rank metric code} $\C$ is a $k$-dimensional $\F_{q^m}$-subspace of $\F_{q^m}^\mathbf{n}$ endowed with the sum-rank metric. 
The minimum sum-rank distance of $\C$ is the integer
\[
d(\C)=\min\{d_{SR}(x,y) \,:\, x, y \in \C, x\neq y  \}.
\]
If the minimum distance is not relevant, we will write that $\C$ is an $[\mathbf{n},k]_{q^m/q}$ code.
\end{definition}

The isometries of $(\F_{q^m}^{\mathbf{n}},d_{SR})$ are also fully characterized (see \cite[Theorem 3.7]{alfarano2021sum} and \cite[Theorem 2]{martinezpenas2021hamming}.).

\begin{theorem}
    Let $\phi:\F_{q^m}^{\mathbf{n}}\rightarrow\F_{q^m}^{\mathbf{n}}$ be an $\F_{q^m}$-linear isometry of $(\F_{q^m}^{\mathbf{n}},d_{SR})$, then there exist $\gamma_1,\ldots,\gamma_t\in\F_{q^m}^*$, $A_i\in\GL(n_i,q)$ for every $i\in\{1,\ldots,t\}$, and $\sigma\in\Sym(t)$ such that $\phi((x_1,\ldots,x_t))=(\gamma_1A_1x_{\sigma(1)},\ldots,\gamma_tA_tx_{\sigma(t)})$, for every $(x_1,\ldots,x_t) \in \F_{q^m}^{\mathbf{n}}$.
\end{theorem}

As a consequence, we give the definition of equivalent codes.

\begin{definition}
    Let $\mc{C},\mc{D}\subseteq\F_{q^m}^{\mathbf{n}}$ be two sum-rank metric codes. The codes $\mc{C},\mc{D}$ are said to be \textbf{equivalent} if there exists an $\mathbb{F}_{q^m}$-linear isometry $\phi$ of $(\F_{q^m}^{\mathbf{n}},d_{SR})$ such that $\phi(\mc{C})=\mc{D}$.
\end{definition}
Similarly as for the Hamming metric and for the rank metric, one can prove a Singleton-like bound also for sum-rank metric codes; see \cite{byrne2021fundamental}.

\begin{theorem}
\label{th:Singletonboundsumrank}
     Let $\mathcal{C}$ be an $[\mathbf{n},k,d]_{q^m/q}$ sum-rank metric code. Then 
\begin{equation}\label{eq:boundgensumrank}
    d \leq N-k+1.
\end{equation}
\end{theorem}

An $[\mathbf{n},k,d]_{q^m/q}$ code is called \textbf{Maximum Sum-Rank Distance code} (or in a short form \textbf{MSRD code}) if its parameters reach the equality in the bound \eqref{eq:boundgensumrank}.

As we have done for rank-metric codes we review also for sum-rank metric codes their interpretation as systems, thus making explicit their geometry; see \cite{neri2021geometry}. First of all observe that each element in $\F_{q^m}^{\mathbf{n}}$ is a tuple of vectors of various dimensions, thus it can be seen by juxtaposition as a vector in $\F_{q^m}^N$. Therefore, we can consider to every $[\mathbf{n},k]_{q^m/q}$ sum-rank metric code a generator matrix $G=(G_1|\ldots|G_t)\in\F_{q^m}^{k\times N}$, where $G_i\in\F_{q^m}^{k\times n_i}$ for every $i\in\{1,\ldots,t\}$.
Also, we will only consider codes that are \textbf{non-degenerate}, that is, if the columns of every $G_i$ are $\fq$-linearly independent.

\begin{theorem}(\hspace{-0.01cm}\cite[Theorem 3.1]{neri2021geometry})
Let $\mathcal{C}$ be a  non-degenerate $[\mathbf{n},k,d]_{q^m/q}$ sum-rank metric code with generator matrix $G=(G_1|\ldots|G_t)$.
Let $\mathcal{U}_i$ be the $\F_q$-span of the columns of $G_i$, for $i\in \{1,\ldots,t\}$. Then, for every $v\in \Fm^k$ we have
\begin{equation}\label{eq:weight_dimension}
\w(v G) = N - \sum_{i=1}^t \dim_{\fq}(\mathcal{U}_i \cap v^{\perp}).\end{equation}
In particular,
\begin{equation}
d=N - \max\left\{ \sum_{i =1}^t\dim_{\fq}(\mathcal{U}_i \cap H)  \colon H\mbox{ is an } \F_{q^m}\mbox{-hyperplane of }\F_{q^m}^k  \right\}.    
\end{equation}
\end{theorem}

The above theorem allows us to consider the notion of systems also in this framework.

\begin{definition}
Let $\mathcal{U}_i$ be an $\F_q$-subspace of $\F_{q^m}^k$ for any $i\in \{1,\ldots,t\}$, such that
$\langle \mathcal{U}_1, \ldots, \mathcal{U}_t \rangle_{\F_{q^m}}=\F_{q^m}^k$. We say that  $\mathcal{U}=(\mathcal{U}_1,\ldots,\mathcal{U}_t)$ is an $[\mathbf{n},k,d]_{q^m/q}$-\textbf{system} if $n_i=\dim_{\F_q}(\mc{U}_i)$ and 
$d=N-\max\left\{\sum_{i=1}^t\dim_{\F_q}(\mathcal{U}_i\cap H) \,:\, H \text{ is an }\F_{q^m}\text{-hyperplane of }\F_{q^m}^k\right\}.$

When not considering parameters explicitly, we will refer to $\mc{U}$ as a $q$-system. Moreover, two $[\mathbf{n},k,d]_{q^m/q}$-systems $\mc{U}$ and $\mathcal{V}$ are \textbf{equivalent} if there exists an $\F_{q^m}$-linear isomorphism $\phi:\mathbb{F}_{q^m}^k\rightarrow\mathbb{F}_{q^m}^k$ such that $\phi(\mc{U})=\mc{V}$.
\end{definition}

\begin{remark}
Directly from the previous definition we have that two $[\mathbf{n},k,d]_{q^m/q}$-systems $\mc{U}$ and $\mc{V}$ are equivalent if there exist $A\in\GL(k,q^m)$, some constants $\gamma_1,\ldots,\gamma_t\in \Fm^*$ and a permutation $\sigma\in\Sym(t)$, such that $\mathcal{U}_i=\gamma_i\mathcal{V}_{\sigma(i)}A$ for every $i\in\{1,\ldots,t\}$.
We point out that the matrix $A$ does not depend on $i$.
\end{remark}

As for rank-metric codes, there is a one-to-one correspondence between equivalence classes of sum-rank metric codes and equivalence classes of systems.

\begin{theorem}(\hspace{-0.01cm}\cite[Theorem 3.7]{neri2021geometry})
\label{th:correspsum}
 Let $\mathfrak{C}[\mathbf{n},k,d]_{q^m/q}$ be the set of equivalence classes of sum-rank nondegenerate $[\mathbf{n},k,d]_{q^m/q}$-codes and $\mathfrak{U}[\mathbf{n},k,d]_{q^m/q}$ the set of equivalence classes of $[\mathbf{n},k,d]_{q^m/q}$-systems. There exists a one-to-one correspondence between $\mathfrak{C}[\mathbf{n},k,d]_{q^m/q}$ and $\mathfrak{U}[\mathbf{n},k,d]_{q^m/q}$.   
\end{theorem}

Thanks to Theorem~\ref{th:correspsum}, sum-rank metric codes and systems can be studied interchangeably. Therefore, in the following, we will freely move between these two perspectives, choosing the one that is more convenient for the problem at hand.

\subsection{Linearized Reed-Solomon codes}

In this section, we will introduce the family of linearized Reed-Solomon codes. The construction of linearized Reed-Solomon codes was first given in \cite{Martinez2018skew}, where the author described them in the skew polynomial setting. We will use the approach introduced in the work of Neri in \cite{neri2022twisted} in the linearized polynomial language.

Let us introduce the polynomial ring $(\mathbb{F}_{q^m}[Y;q^s],+,\cdot)$ and link it with $\tilde{\mc{L}}_{m,q^s}$. By $\mathbb{F}_{q^m}[Y;q^s]$ we consider the set of polynomials in the indeterminate $Y$ with coefficients in $\mathbb{F}_{q^m}$, where the sum is the usual addition of polynomials, and the multiplication follows the rule
\begin{equation*}
    Ya=a^{q^s}Y, \text{ for any }a\in\mathbb{F}_{q^m}
\end{equation*}
extended by associativity and distributivity. It can be checked that $\mathbb{F}_{q^m}[Y;q^s]$ is commutative if and only if $s\equiv0\pmod m$ and its center is $\mathbb{F}_q[Y^m]$. 

\begin{theorem}{(\hspace{-0.01cm}\cite[Theorem 2.6]{neri2022twisted})}
Let
\begin{align*}
    \Phi\,:\,\qquad\qquad\mathbb{F}_{q^m}[Y;q^s]\qquad\qquad&\rightarrow\qquad\qquad\tilde{\mc{L}}_{m,q^s}\\
    f_0+f_1Y+\ldots+f_dY^d\quad&\mapsto\quad f_0x+f_1x^{q^s}+\ldots+f_dx^{q^{sd}}.
\end{align*}
    Then $\Phi$ is a $\mathbb{F}_q$-algebra surjective homomorphism whose kernel is the two-sided ideal $(Y^m-1)$. Consequently,
    \begin{equation*}
        \mathbb{F}_{q^m}[Y;q^s]/(Y^m-1)\cong\tilde{\mc{L}}_{m,q^s}.
    \end{equation*}
\end{theorem}

Observe that we can define the evaluation of a polynomial $\mathbb{F}_{q^m}[Y;q^s]/(Y^m-1)$ in an element $\beta\in\mathbb{F}_{q^m}$ as the evaluation of $\Phi(F)$ in $\beta$, and with this meaning we will write $F(\beta)$.

\begin{definition}
Let $s$ be an integer coprime with $m$, the \textbf{norm} of $\alpha$ with respect to $\mathbb{F}_{q^m}/\mathbb{F}_q$ is defined by
    \begin{equation*}
        \mathrm{N}_s(\alpha):=\prod_{i=0}^{m-1}\alpha^{q^{si}}.
    \end{equation*}
We also define the \textbf{$j$-th truncated norm} for every $\alpha\in\mathbb{F}_{q^m}$ as
\begin{equation*}
\mathrm{N}_s^j(\alpha):=
    \begin{cases}
    1&j=0,\\
\prod_{i=0}^{j-1}\alpha^{q^{si}}&j\geq1.
    \end{cases}
\end{equation*}
\end{definition}

Let $\alpha_1,\ldots,\alpha_t\in\mathbb{F}_{q^m}^*$ be elements with pairwise distinct norms. Let $\Lambda=\{\,\lambda_1,\ldots,\lambda_t\,\}\subseteq\mathbb{F}_q^*$ where $\lambda_i=\mathrm{N}_s(\alpha_i)$. Define
\begin{equation*}
    H_{\Lambda}(Y)=\prod_{i=1}^t(Y^{m}-\lambda_i)\in\mathbb{F}_{q^m}[Y;q^s].
\end{equation*}
By definition $H_{\Lambda}(Y)$ belongs to $\mathbb{F}_q[Y^m]$ which is the center of $\mathbb{F}_{q^m}[Y;q^s]$, as it is the product of $t$ central polynomials. Then $H_{\Lambda}(Y)$ generates a two-sided ideal.

Given a skew polynomial $F(Y)=f_0+f_1Y+\ldots+f_dY^d$ and $\alpha\in\mathbb{F}_{q^m}^*$ we denote by $F_\alpha$ the polynomial obtained from $F$ as
\begin{equation*}
    F_{\alpha}(Y):=\sum_{i=0}^d f_i\mathrm{N}_s^i(\alpha)Y^i.
\end{equation*}
With this notation, we can generalize the previous result on $\Phi$ to the following.

\begin{theorem}{(\hspace{-0.01cm}\cite[Theorem 4.1]{neri2022twisted})}
With the previous notation, the map
\begin{align*}
    \Phi_{\alpha}\,:\,\mathbb{F}_{q^m}[Y;q^s]\quad&\rightarrow\quad(\tilde{\mc{L}}_{m,q^s})^t\\
    F(Y)\quad&\mapsto\quad (\Phi(F_{\alpha_1}),\ldots,\Phi(F_{\alpha_t})).
\end{align*}
is a surjective $\mathbb{F}_q$-algebra homomorphism, whose kernel is $(H_{\Lambda}(Y))$. Hence, it induces an $\mathbb{F}_q$-algebra isomorphism
\begin{equation*}
    \mathbb{F}_{q^m}[Y;q^s]/(H_{\Lambda}(Y))\cong(\tilde{\mc{L}}_{m,q^s})^t.
\end{equation*}
\end{theorem}
 
Finally, we are ready to define the linearized Reed-Solomon codes.

\begin{definition}
Let $s \in \{1,\ldots,m\}$ and $k\leq tm$ with $\gcd(s,m)=1$.
Let $\alpha=(\alpha_1,\ldots,\alpha_t)\in\F_{q^m}^t$ be a tuple of nonzero elements with pairwise distinct norms over $\fq$. 
The \textbf{linearized Reed-Solomon code} of dimension $k$ is the code
\begin{equation*}
  LR_k^{q^s}[\alpha]=\{(F_{\alpha_1}(x),\ldots,F_{\alpha_t}(x)):F\in\tilde{\mc{L}}_{m,q^s}, \deg(F)_{q^s}\leq k-1\}.
\end{equation*}
\end{definition}

Let us now show that this type of codes can be seen as evaluation codes in $\F_{q^m}^{\mathbf{n}}$, referring to \cite{neri2021geometry}. In particular, take $\beta=(\beta_1,\ldots,\beta_n)\in\F_{q^m}^n$ and define the following map
\begin{align*}
\mathrm{ev}_{\beta}:\tilde{\mc{L}}_{m,q^s}&\longrightarrow\F_{q^m}^n\\
    f&\longmapsto(f(\beta_1),\ldots,f(\beta_n))
\end{align*}
With this notation, we can give the definition of linearized Reed-Solomon codes as evaluation codes.
For our purposes, we will only define linearized Reed-Solomon codes $\fqm^{\mathbf{n}}$ where $n_1=\ldots=n_t$.

\begin{definition}
Let $s \in \{1,\ldots,m\}$ and $k\leq tm$ with $\gcd(s,m)=1$.
Let $\mathbf{n}=(n,\ldots,n)$ with $n\leq m$ and let $\alpha=(\alpha_1,\ldots,\alpha_t)\in\F_{q^m}^t$ be a tuple of nonzero elements with pairwise distinct norms over $\fq$. 
Let $\beta=(\beta_1,\ldots,\beta_n)\in\F_{q^m}^{n}$, where $\beta_1,\ldots,\beta_n$ are $\F_q$-linearly independent, and define
\begin{equation*}
  LR_k^{q^s}[\alpha,\beta]=\{(\mathrm{ev}_{\beta}(F_{\alpha_1}),\ldots,\mathrm{ev}_{\beta}(F_{\alpha_t})):F\in\tilde{\mc{L}}_{m,q^s}, \deg(F)_{q^s}\leq k-1\}\subseteq\F_{q^m}^{\mathbf{n}}.
\end{equation*}
Then $LR_k^{q^s}[\alpha,\beta]$ is an $[\mathbf{n},k]_{q^m/q}$ linearized Reed-Solomon code. Moreover, whenever $\beta_1,\ldots,\beta_n$ are a standard basis of $\F_{q^m}^{n}$, the notation is softened to $LR_k^{q^s}[\alpha]$.
\end{definition}

Linearized Reed-Solomon codes were the first examples of MSRD codes, as proved by Mart{\'\i}nez-Pe{\~n}as in \cite{Martinez2018skew}.

\begin{theorem}{(\hspace{-0.01cm}\cite[Theorem 4]{Martinez2018skew})}
Let $s \in \{1,\ldots,m\}$ and $k\leq tm$ with $\gcd(s,m)=1$.
Let $\mathbf{n}=(n,\ldots,n)$ with $n\leq m$ and let $\alpha=(\alpha_1,\ldots,\alpha_t)\in\F_{q^m}^t$ be a tuple of nonzero elements with pairwise distinct norms over $\fq$. 
Let $\beta=(\beta_1,\ldots,\beta_n)\in\F_{q^m}^{n}$, where $\beta_1,\ldots,\beta_n$ are $\F_q$-linearly independent, then $LR_k^{q^s}[\alpha,\beta]$ is an MSRD code in $\F_{q^m}^{\mathbf{n}}$.
\end{theorem}

The same result holds if the evaluation vector $\beta$ is changed in each block; however, for our purposes, we have stated a simpler version.

Finally, before concluding this section, we show a $q$-system associated with linearized Reed-Solomon codes; see \cite{santonastaso2022subspace}.

\begin{remark}
Consider for $i\in\{1,\ldots,t\}$ the $\mathbb{F}_q$-subspaces $S=\la\beta_1,\ldots,\beta_{n}\ra_{\mathbb{F}_q}\subseteq\mathbb{F}_{q^m}$ such that $\dim_{\mathbb{F}_q}(S)=n$ then
\begin{equation*}
    G=\begin{pmatrix}
        G_1&|&\ldots&|&G_t
    \end{pmatrix},
\end{equation*}
where, for every $i\in\{1,\ldots,t\}$,
\begin{equation*}
G_i=
    \begin{pmatrix}
        \beta_1&\ldots&\beta_{n}\\  \alpha_i(\beta_1)^{q^s}&\ldots&\alpha_i(\beta_{n})^{q^s}\\
        \vdots&&\vdots\\
        \alpha_i^{1+\ldots +q^{s(k-2)}}(\beta_1)^{q^{s(k-1)}}&\ldots&\alpha_i^{1+\ldots +q^{s(k-2)}}(\beta_{n})^{q^{s(k-1)}}
    \end{pmatrix}.
\end{equation*}
For every $i\in\{1,\ldots,t\}$ we can write the $q$-system associated with $G_i$ as
\begin{align*}
    \mc{U}_i&=\la(\beta_1,\mathrm{N}_s^1(\alpha_i){(\beta_1)}^{q^s},\ldots,\mathrm{N}_s^{k-1}(\alpha_i){(\beta_1)}^{q^{s(k-1)}}),\ldots,(\beta_{n},\mathrm{N}_s^1(\alpha_i){(\beta_{n})}^{q^s},\ldots,\mathrm{N}_s^{k-1}(\alpha_i){(\beta_{n})}^{q^{s(k-1)}})\ra_{\mathbb{F}_q}\\
    &=\{(x,\mathrm{N}_s^1(\alpha_i)x^{q^s},\ldots,\mathrm{N}_s^{k-1}(\alpha_i)x^{q^{s(k-1)}}):x\in S\}.
\end{align*}
Therefore, the $q$-system associated with $G$ is the tuple $(\mc{U}_1,\ldots,\mc{U}_t)$.
\end{remark}

\begin{remark}
Consider $\mc{U}=(\mc{U}_1,\ldots,\mc{U}_t)$ a linearized Reed-Solomon system and observe that, supposing $k\leq m$, every component $\mc{U}_i$ is equivalent to a Gabidulin system, since\footnote{Note that, with an abuse of notation, we always consider the matrix multiplication on the left. Hence elements of $\fqm^k$ will be converted to column vectors whereas after the multiplication these revert to row vectors. This principle will be used throughout this article.}
\begin{equation*}
    \mc{U}_i=\begin{pmatrix}
        1&0&\dots&0\\
        0&\mathrm{N}_s^1(\alpha_i)&&\vdots\\
        \vdots&&\ddots&0\\
        0&\dots&0&\mathrm{N}_s^{k-1}(\alpha_i)
    \end{pmatrix}G_{k,s,m}.
\end{equation*}
Here, the matrix with the norms is invertible, because every $\alpha_i$ is non-zero.
\end{remark}

\section{The stabilizer of the Gabidulin system}\label{sec:4}

Let both $S,T\subseteq\mathbb{F}_{q^m}$ be $\mathbb{F}_q$-subspaces of $\mathbb{F}_{q^m}$. Suppose that $s,s'\in\{1,\ldots,m\}$ such that $\gcd(s,m)=\gcd(s',m)=1$, and, $G_{k,s,m}[T]$ and $G_{k,s',m}[S]$ are two $q$-systems associated with Gabidulin codes. Observe that, as a consequence of Theorem \ref{th:corresp} and \cite[Theorem 4.8]{lunardon2018generalized}, we have the following.

\begin{theorem}\label{thm:Gabs-s}
    Let $1<k < m-1$. Suppose that $s$ and $s'$ are such that $\gcd(s,m)=\gcd(s',m)=1$. Then $G_{k,s,m}$ and $G_{k,s',m}$ are equivalent if and only if $s'\equiv\pm s\pmod{m}$.
    If $k \in \{1,m-1,m\}$, then $G_{k,s,m}$ and $G_{k,s',m}$ are equivalent for any $s$ and $s'$ coprime with $m$.
\end{theorem}

As a consequence, one can derive the number of inequivalent Gabidulin codes; see e.g. \cite[Theorem 5.6]{neri2020equivalence}.

\begin{corollary}\label{cor:numbinequiGabidulin}
    The number of inequivalent Gabidulin codes in $\fqm^m$ of dimension $k$ equals $\varphi(m)/2$ if $1<k<m-1$ and one otherwise.
\end{corollary}

The same result does not hold when considering $G_{k,s,m}[S]$ and $G_{k,s',m}[T]$ with $S$ and $T$ different from $\fqm$.
Indeed, already in the case in which $s=s'$, we can prove the following (where the first part can also be derived from Theorem \ref{th:connection} and \cite[Theorem 2]{berger2003isometries}).

\begin{theorem}
\label{th:eqgab}
Let $S,T\subseteq\mathbb{F}_{q^m}$ be two $\mathbb{F}_q$-subspaces of $\mathbb{F}_{q^m}$ and $s\in\{1,\ldots,m\}$ such that $\gcd(s,m)=1$. If $\dim_{\mathbb{F}_q}(S)=\dim_{\mathbb{F}_q}(T)>k$, then $G_{k,s,m}[S]$ and $G_{k,s,m}[T]$ are equivalent if and only if there exists $\alpha\in\mathbb{F}_{q^m}^*$ such that $T=\alpha S$.
In particular,
    \begin{equation*}
    \stab_{\mathrm{GL}(k,q^m)}(G_{k,s,m})=
    \left\{
    \begin{pmatrix}
d & 0 & \cdots & 0 \\
0 & d^{q^s} & \cdots & 0 \\
\vdots & \vdots & \ddots & \vdots \\
0 & 0 & \cdots & d^{q^{s(k-1)}}
\end{pmatrix}\,:\,d\in\F_{q^m}^*\right\}.
\end{equation*}
\end{theorem}
\begin{proof}
By definition, $G_{k,s,m}[T]$ and $G_{k,s,m}[S]$ are equivalent if and only if there exists $A\in\GL(k,q^m)$ with the property that for every $x\in S$ there exists $t\in T$ such that
\begin{equation}
\label{eq:GLeq}
    A\begin{pmatrix}
        x\\
        x^{q^s}\\
        \vdots\\
        x^{q^{s(k-1)}}
    \end{pmatrix}=
    \begin{pmatrix}
        t\\
        t^{q^s}\\
        \vdots\\
        t^{q^{s(k-1)}}
    \end{pmatrix}.
\end{equation}
Let $\textbf{a}_1,\ldots,\textbf{a}_k$ be the rows of the matrix $A$.
Equation \eqref{eq:GLeq} can be written as follows,
\begin{equation}
\label{eq:system}
\begin{cases}
    \langle\textbf{a}_1,(x,\ldots,x^{q^{s(k-1)}})\rangle=t,\\
    \langle\textbf{a}_2,(x,\ldots,x^{q^{s(k-1)}})\rangle=t^{q^s},\\
    \qquad\qquad\vdots\\
    \langle\textbf{a}_k,(x,\ldots,x^{q^{s(k-1)}})\rangle=t^{q^{s(k-1)}},
\end{cases}
\end{equation}
where $\langle \cdot, \cdot \rangle$ denotes the standard inner product.
Our aim is to prove that $A$ has the following form
\begin{equation*}
A=\begin{pmatrix}
    a_{1,1}&\dots&0\\
    \vdots&\ddots&\vdots\\
    0& \dots &a_{1,1}^{q^{s(k-1)}}
\end{pmatrix}.
\end{equation*}
We prove this by induction on the rows of $A$. We start by considering the first two equations of \eqref{eq:system}, from which we derive
\begin{equation*}
    \la \textbf{a}_1^{q^s},(x^{q^s},\ldots,x^{q^{sk}})\ra-\la \textbf{a}_2,(x,\ldots,x^{q^{s(k-1)}})\ra=0.
\end{equation*}
This leads to the following expression
\begin{equation}\label{eq_Th4.3}
   a_{1,k}^{q^s}x^{q^{sk}}+\sum_{j=1}^{k-1}(a_{1,j}^{q^s}-a_{2,j+1})x^{q^{si}} - a_{2,1}x=0,
\end{equation}
for every $x \in S$.
Since $\dim_{\mathbb{F}_q}(S)>k$, by Theorem \ref{thm:boundnumbroots}, we have that both $a_{1,k}^{q^s}X^{q^{sk}}+\sum_{j=1}^{k-1}(a_{1,j}^{q^s}-a_{2,j+1})X^{q^{si}} - a_{2,1}X$ and polynomial \eqref{eq_Th4.3} are both the zero polynomial. Hence, we can conclude that
\begin{equation*}
    \begin{cases}
        a_{1,k}=0,\\
        a_{2,1}=0,\\
        a_{2,j}=a_{1,j-1}^{q^s},&j=2,\ldots,k
    \end{cases}
\end{equation*}
thus, the matrix $A$ is
\begin{equation*}
    \begin{pmatrix}
        a_{1,1}&a_{1,2}&\ldots&a_{1,k-1}&0\\
        0&a_{1,1}^{q^s}&\ldots&a_{1,k-2}^{q^s}&a_{1,k-1}^{q^s}\\
        \vdots&&\ldots&&\vdots\\
        a_{k,1}&&\ldots&&a_{k,k}
    \end{pmatrix}.
\end{equation*}
Now, suppose that for every $2\leq\bar{i}\leq i-1$ we have
\begin{equation*}
    \begin{cases}
        a_{1,k+2-\bar{i}}=0,\\
        a_{\bar i,j}=0,&j=1,\ldots\bar i-1\\
        a_{\bar i,j}=a_{1,j-\bar i +1}^{q^{s(\bar i-1)}},&j=\bar i,\ldots, k
    \end{cases},
\end{equation*}
i.e., the matrix $A$ is
\begin{equation*}
    \begin{pmatrix}
        a_{1,1}&a_{1,2}&\ldots&a_{1,k-i+1}&0&0&\ldots&0\\
        0&a_{1,1}^{q^s}&\ldots&a_{1,k-i}^{q^s}&a_{1,k-i+1}^{q^s}&0&\ldots&0\\
        \vdots&&&\ldots&&&&\vdots\\
        0&\ldots&0&a_{1,1}^{q^{s(i-2)}}&\ldots&&&a_{1,k-i+2}^{q^{s(i-2)}}\\
        a_{i,1}&&&&\ldots&&&a_{i,k}\\
        \vdots&&&&\ldots&&&\vdots\\        
        a_{k,1}&&&&\ldots&&&a_{k,k}
    \end{pmatrix}.
\end{equation*}

Substituting the first equation of System \eqref{eq:system} into the $i$-th one, we get
\begin{equation*}
    \la \textbf{a}_1^{q^{s(i-1)}},(x^{q^{s(i-1)}},\ldots,x^{q^{s(k-1)+s(i-1)}})\ra-\la \textbf{a}_i,(x,\ldots,x^{q^{s(k-1)}})\ra=0,
\end{equation*}
thus, we have
\begin{equation*}
   \sum_{j=1}^{k}a_{1,j}^{q^{s(i-1)}}x^{q^{s(i-1)+s(j-1)}}-\sum_{j=1}^{k}a_{i,j}x^{q^{s(j-1)}}=0.
\end{equation*}
Using the inductive hypothesis and simplifying, we conclude that
\begin{equation*}
   a_{1,k-i+2}^{q^{s(i-1)}}x^{q^{sk}}+\sum_{j=i}^{k}(a_{1,j-i+1}^{q^{s(i-1)}}-a_{i,j})x^{q^{{s(j-1)}}} -\sum_{j=1}^{i-1}a_{i,j}x^{q^{{s(j-1)}}}=0,
\end{equation*}
for every $x \in S$.
Arguing as before, we obtain the following conditions
\begin{equation*}
    \begin{cases}
        a_{1,k+2-i}=0,\\
        a_{i,j}=0,&j=1,\ldots,i-1\\
        a_{i,j}=a_{1,j-i +1}^{q^{s(i-1)}},&j=i,\ldots, k
    \end{cases}.
\end{equation*}
Iterating this argument, after the $k$ steps, we get that the matrix $A$ is
\begin{equation*}
A=\begin{pmatrix}
    a_{1,1}&\dots&0\\
    \vdots&\ddots&\vdots\\
    0& \dots &a_{1,1}^{q^{s(k-1)}}
\end{pmatrix}.
\end{equation*}
Hence, there exists $\alpha\in\mathbb{F}_{q^m}^*$ such that for every $x\in S$ there exists $t\in T$ with $\alpha s=t$, thus $\alpha S=T$. 

Now, suppose there exists $\alpha\in\mathbb{F}_{q^m}^*$ such that $\alpha S=T$, then for every $x\in S$ there exists $t\in T$ such that $\alpha x=t$ and we can write
\begin{equation*}
    \begin{pmatrix}
        \alpha x\\
        (\alpha x)^{q^s}\\
        \vdots\\
        (\alpha x)^{q^{s(k-1)}}
    \end{pmatrix}=
    \begin{pmatrix}
        t\\
        t^{q^s}\\
        \vdots\\
        t^{q^{s(k-1)}}
    \end{pmatrix},
\end{equation*}
so that $G_{k,s,m}[T]$ is equivalent to $G_{k,s,m}[S]$ via the matrix
\begin{equation*}
    A=\begin{pmatrix}
       \alpha &0&\ldots&0\\
       0&\alpha^{q^s}&\ldots&0\\
       \vdots&&\ddots&\vdots\\
       0&\ldots&0&\alpha^{q^{s(k-1)}}
    \end{pmatrix}.
\end{equation*}
The stabilizer is obtained by considering the maps between $G_{k,s,m}[S]$ and $G_{k,s,m}[T]$ with $S=T=\fqm$.
\end{proof}

When $k=m$, the stabilizer of $G_{m,s,m}$ in $\mathrm{GL}(m,q^m)$ is different and is isomorphic to $\mathrm{GL}(m,q)$.

\begin{theorem}
\label{th:eqgab=}
    Let $s\in\{1,\ldots,m\}$ be such that $(s,m)=1$. Then
    \begin{equation*}
    \stab_{\mathrm{GL}(m,q^m)}(G_{m,s,m}) =\left\{
    D_f \colon 
f\in\tilde{\mc{L}}_{m,q^s} \text{ and }
f \text{ is invertible}
\right\}.
    \end{equation*}
\end{theorem}
\begin{proof}
An element $A\in\GL(m,q^m)$ is in $\stab_{\mathrm{GL}(k,q^m)}(G_{m,s,m})$ if and only if for every $x\in \F_{q^m}$ there exists $t\in \F_{q^m}$ that satisfies
\begin{equation*}
    A\begin{pmatrix}
        x\\
        x^{q^s}\\
        \vdots\\
        x^{q^{s(m-1)}}
    \end{pmatrix}=
    \begin{pmatrix}
        t\\
        t^{q^s}\\
        \vdots\\
        t^{q^{s(m-1)}}
    \end{pmatrix}.
\end{equation*}
As before, the above condition is equivalent to the following system
\begin{equation*}
\begin{cases}
    \langle\textbf{a}_1,(x,\ldots,x^{q^{s(m-1)}})\rangle=t,\\
    \langle\textbf{a}_2,(x,\ldots,x^{q^{s(m-1)}})\rangle=t^{q^s},\\
    \qquad\qquad\vdots\\
    \langle\textbf{a}_m,(x,\ldots,x^{q^{s(m-1)}})\rangle=t^{q^{s(m-1)}}
,\end{cases}
\end{equation*}
where $\mathbf{a}_1,\ldots,\mathbf{a}_m$ are the rows of $A$.
Let $\mathbf{a}_1=(a_0,\ldots,a_{m-1})$, then
\begin{equation*}
    t=\sum_{k=0}^{m-1}a_kx^{q^{sk}},
\end{equation*}
by plugging the expression of $t$ we obtain

\begin{equation*}
    \sum_{k=0}^{m-1}a_{k}^{q^{s(i-1)}}x^{q^{s(k+i-1)}}=\sum_{j=0}^{m-1}a_{i,j+1}x^{q^{sj}},
\end{equation*}
for every $i\in\{1,\ldots,m\}$.
The above condition can be rewritten as follows
\begin{equation*}
    \sum_{j=i-1}^{m+i-2}a_{j-i+1}^{q^{s(i-1)}}x^{q^{sj}}=\sum_{j=0}^{m-1}a_{i,j+1}x^{q^{sj}},
\end{equation*}
and considering the relation $x^{q^{s(m+l)}}=x^{q^{sl}}$ for every $l\geq0$ and $x \in \fqm$, we can rewrite the previous equation as
\begin{equation*}
    \sum_{j=0}^{i-2}a_{m+j-i+1}^{q^{s(i-1)}}x^{q^{sj}}+\sum_{j=i-1}^{m-1}a_{j-i+1}^{q^{s(i-1)}}x^{q^{sj}}=\sum_{j=0}^{m-1}a_{i,j+1}x^{q^{sj}}
\end{equation*}
thus,
\begin{equation*}
    \begin{cases}
        a_{i,j+1}=a_{m+j-i+1}^{q^{s(i-1)}}\quad j\in\{0,\ldots,i-2\},\\
        a_{i,j+1}=a_{j-i+1}^{q^{s(i-1)}}\quad j\in\{i-1,\ldots,m-1\},
    \end{cases}
\end{equation*}
for every $i\in\{2,\ldots,m\}$. Therefore, $A=D_f$, where $f=a_0x+\ldots+a_{m-1}x^{q^{s(m-1)}}\in\tilde{\mc{L}}_{m,q^s}$.
\end{proof}

Now let us consider $s' \equiv -s \pmod m$. In this case, it is readily verified that the following hold.

\begin{proposition}
\label{le:gab-s}
    Let $s\in\{1,\ldots,m\}$ be such that $\gcd(s,m)=1$ and let $T\subseteq\F_{q^m}$ be a subspace of $\fqm$, then
    \begin{equation*}
        G_{k,-s,m}[T]= \begin{pmatrix}
        0&\ldots&1\\
        \vdots&\iddots&\vdots\\
        1&\ldots&0
    \end{pmatrix}G_{k,s,m}[T^{q^{-s(k-1)}}].
    \end{equation*}
\end{proposition}

\begin{remark}
By the previous proposition, when considering Gabidulin systems, we have that $G_{k,-s,m}$ is equivalent to $G_{k,s,m}$ and so, when studying the equivalence of $G_{k,s,m}$, we may just consider the case where $s<m/2$.
\end{remark}

\section{Equivalence classes of linearized Reed-Solomon codes}\label{sec:5}

In this section, we focus on characterizing equivalent linearized Reed-Solomon codes, which will be essential to count the number of inequivalent classes of linearized Reed-Solomon codes. Our approach is geometric in nature, hence we will count the number of inequivalent  systems associated with linearized Reed-Solomon codes.
To this aim, we consider $(\mc{U}_1,\ldots,\mc{U}_t)$ and $(\mc{V}_1,\ldots,\mc{V}_t)$ two $[\mathbf{m},k]_{q^m/q}$-systems associated with linearized Reed-Solomon codes, where $\mathbf{m}=(m,\ldots,m)$, i.e. for every $i, j \in \{1,\ldots t\}$
\begin{align*}
    \mc{U}_i&=\{(x,\mathrm{N}_s^1(\alpha_i)x^{q^s},\ldots,\mathrm{N}_s^{k-1}(\alpha_i)x^{q^{s(k-1)}})\,:\,x\in\mathbb{F}_{q^m}\}\\
    \mc{V}_j&=\{(y,\mathrm{N}_{s'}^1(\beta_j)y^{q^{s'}},\ldots,\mathrm{N}_{s'}^{k-1}(\beta_j)y^{q^{s'(k-1)}})\,:\,y\in\mathbb{F}_{q^m}\},
\end{align*}
for some $s,s' \in \{1,\ldots,m\}$ with $\gcd(s,m)=\gcd(s',m)=1$ and both $\{\alpha_1,\ldots,\alpha_t\}$ and $\{\beta_1,\ldots,\beta_t\}$ subsets of $\fqm$ with pairwise distinct nonzero norms over $\fq$.

\begin{remark}\label{rem:equivGksm}
    We observe that all elements of the $q$-systems $(\mc{U}_1,\ldots,\mc{U}_t)$ and $(\mc{V}_1,\ldots,\mc{V}_t)$ are equivalent to the Gabibulin $q$-systems, that is, $\mc{U}_i$ is equivalent to $G_{k,s,m}$ and $\mc{V}_j$ is equivalent to $G_{k,s',m}$ for every $i$ and $j$.
\end{remark}

As a consequence of the above remark and Theorem \ref{thm:Gabs-s} we have the following.

\begin{proposition}\label{prop:s<m/2}
    Let $k\leq m$ and let $s,s' \in \{1,\ldots,m\}$ with $\gcd(s,m)=\gcd(s',m)=1$ and $\{\alpha_1,\ldots,\alpha_t\}$ and $\{\beta_1,\ldots,\beta_t\}$ subsets of $\fqm$ whose elements have pairwise distinct nonzero norms over $\fq$.
    Consider $(\mc{U}_1,\ldots,\mc{U}_t)$ and $(\mc{V}_1,\ldots,\mc{V}_t)$ where
\begin{align*}
    \mc{U}_i&=\{(x,\mathrm{N}_s^1(\alpha_i)x^{q^s},\ldots,\mathrm{N}_s^{k-1}(\alpha_i)x^{q^{s(k-1)}})\,:\,x\in\mathbb{F}_{q^m}\}\\
    \mc{V}_j&=\{(y,\mathrm{N}_{s'}^1(\beta_j)y^{q^{s'}},\ldots,\mathrm{N}_{s'}^{k-1}(\beta_j)y^{q^{s'(k-1)}})\,:\,y\in\mathbb{F}_{q^m}\},
\end{align*}
for every $i, j \in \{1,\ldots t\}$.
If $(\mc{U}_1,\ldots,\mc{U}_t)$ and $(\mc{V}_1,\ldots,\mc{V}_t)$ are equivalent, then $s'\equiv \pm s \pmod{m}$.
\end{proposition}
\begin{proof}
    If $(\mc{U}_1,\ldots,\mc{U}_t)$ and $(\mc{V}_1,\ldots,\mc{V}_t)$ are equivalent, then by definition there exist $A\in\GL(k,q^m),\,\sigma\in \Sym(t),\,(\gamma_1,\ldots,\gamma_t)\in(\mathbb{F}_{q^m}^*)^t$ such that $\mc{U}_i=\gamma_iA\mc{V}_{\sigma(i)}$ for every $i$.
    By Remark \ref{rem:equivGksm}, we have that $\mc{U}_i$ is equivalent to $G_{k,s,m}$ and $\mc{V}_j$ is equivalent to $G_{k,s',m}$ for every $i$ and $j$.
    As a consequence, we derive that $G_{k,s,m}$ and $G_{k,s',m}$ are equivalent and therefore by Theorem \ref{thm:Gabs-s} it follows that $s'\equiv \pm s \pmod{m}$.
\end{proof}

Now, as for the Gabidulin codes, we will show that we can always assume that we can restrict ourselves to the case $s<m/2$.

\begin{theorem}
Let $k\leq m$ and let $s \in \{1,\ldots,m\}$ with $\gcd(s,m)=1$ and $\{\alpha_1,\ldots,\alpha_t\}$ a subset of $\fqm$ whose elements have pairwise distinct nonzero norms over $\fq$.
Let $(\mc{U}^{-s}_1,\ldots,\mc{U}^{-s}_t)$ and $(\mc{U}^{s}_1,\ldots,\mc{U}^{s}_t)$ be defined as
\begin{align*}
    \mc{U}^{-s}_i&=\{(x,\mathrm{N}_{-s}^1(\alpha_i)x^{q^{-s}},\ldots,\mathrm{N}_{-s}^{k-1}(\alpha_i)x^{q^{-s(k-1)}})\,:\,x\in\mathbb{F}_{q^m}\},\\
    \mc{U}^{s}_i&=\{(x,\mathrm{N}_s^1((\alpha_i^{-1})^{q^{-s(k-2)}})x^{q^s},\ldots,\mathrm{N}_s^{k-1}((\alpha_i^{-1})^{q^{-s(k-2)}})x^{q^{s(k-1)}})\,:\,x\in\mathbb{F}_{q^m}\},
\end{align*}
for every $i\in\{1,\ldots,t\}$, then $(\mc{U}^{-s}_1,\ldots,\mc{U}^{-s}_t)$ is equivalent to $(\mc{U}^{s}_1,\ldots,\mc{U}^{s}_t)$.
\end{theorem}
\begin{proof}
We show that
\begin{equation}
\label{eq:anti1} 
    \mathrm{N}_{-s}^{k-1}(\alpha_i)\begin{pmatrix}
                0&\ldots&1\\
        \vdots&\ddots&\vdots\\
        1&\ldots&0
    \end{pmatrix}\mc{U}_i^s=\mc{U}_i^{-s}
\end{equation}
for every $i\in\{1,\ldots,t\}$, or equivalently
\begin{equation}
\label{eq:anti2}
\begin{pmatrix}
        1&\dots&0\\
        \vdots&\ddots&\vdots\\
        0&\dots&\mathrm{N}_{-s}^{k-1}(\alpha_i^{-1})
    \end{pmatrix}\mathrm{N}_{-s}^{k-1}(\alpha_i)\begin{pmatrix}
                0&\ldots&1\\
        \vdots&\ddots&\vdots\\
        1&\ldots&0
    \end{pmatrix}
    \begin{pmatrix}
        1&\dots&0\\
        \vdots&\ddots&\vdots\\
        0&\dots&\mathrm{N}_s^{k-1}((\alpha_i^{-1})^{q^{-s(k-2)}})
    \end{pmatrix}G_{k,s,m}=G_{k,-s,m},
\end{equation}
for every $i\in\{1,\ldots,t\}$.
Observe that the matrix
\begin{equation*}
\begin{pmatrix}
        1&\dots&0\\
        \vdots&\ddots&\vdots\\
        0&\dots&\mathrm{N}_{-s}^{k-1}(\alpha_i^{-1})
    \end{pmatrix}\mathrm{N}_{-s}^{k-1}(\alpha_i)\begin{pmatrix}
                0&\ldots&1\\
        \vdots&\ddots&\vdots\\
        1&\ldots&0
    \end{pmatrix}
    \begin{pmatrix}
        1&\dots&0\\
        \vdots&\ddots&\vdots\\
        0&\dots&\mathrm{N}_s^{k-1}((\alpha_i^{-1})^{q^{-s(k-2)}})
    \end{pmatrix} 
\end{equation*}
is the antidiagonal matrix with entries equal to 
\begin{equation*}
 \begin{aligned}
    \mathrm{N}_{-s}^{k-1}(\alpha_i)\mathrm{N}_{-s}^{j}(\alpha_i^{-1})\mathrm{N}_s^{k-j-1}((\alpha_i^{-1})^{q^{-s(k-2)}})&=\alpha_i^{1+\ldots+q^{-s(k-2)}-(1+\ldots+q^{-s(j-1)}+q^{-s(k-2)}(1+\ldots+q^{s(k-j-2)}))}\\
    &=\alpha_i^{1+\ldots+q^{-s(k-2)}-(1+\ldots+q^{-s(k-2)})}=1.
\end{aligned}   
\end{equation*}

Thus, 
\[ \begin{pmatrix}
        1&\dots&0\\
        \vdots&\ddots&\vdots\\
        0&\dots&\mathrm{N}_{-s}^{k-1}(\alpha_i^{-1})
    \end{pmatrix}\mathrm{N}_{-s}^{k-1}(\alpha_i)\begin{pmatrix}
                0&\ldots&1\\
        \vdots&\ddots&\vdots\\
        1&\ldots&0
    \end{pmatrix}
    \begin{pmatrix}
        1&\dots&0\\
        \vdots&\ddots&\vdots\\
        0&\dots&\mathrm{N}_s^{k-1}((\alpha_i^{-1})^{q^{-s(k-2)}})
    \end{pmatrix} = \begin{pmatrix}
                0&\ldots&1\\
        \vdots&\ddots&\vdots\\
        1&\ldots&0
    \end{pmatrix}\]
 hence, \eqref{eq:anti2}, and also \eqref{eq:anti1}, is verified.
\end{proof}

Therefore, inequivalent linearized Reed-Solomon codes/systems can be obtained from the family of linearized Reed-Solomon codes/systems in which $s< m/2$.

\begin{theorem}
\label{th:eqsets}
Let $k< m$ and let $s \in \{1,\ldots,m\}$ with $\gcd(s,m)=1$ and $\{\alpha_1,\ldots,\alpha_t\}$ and $\{\beta_1,\ldots,\beta_t\}$ be two subsets of $\fqm$ whose elements have pairwise distinct nonzero norms over $\fq$.
Let $(\mc{U}_1,\ldots,\mc{U}_t)$ and $(\mc{V}_1,\ldots,\mc{V}_t)$ be defined as
\begin{align*}
    \mc{U}_i&=\{(x,\mathrm{N}_s^1(\alpha_i)x^{q^s},\ldots,\mathrm{N}_s^{k-1}(\alpha_i)x^{q^{s(k-1)}})\,:\,x\in\mathbb{F}_{q^m}\},\\
    \mc{V}_j&=\{(y,\mathrm{N}_s^1(\beta_j)y^{q^{s}},\ldots,\mathrm{N}_s^{k-1}(\beta_j)y^{q^{s(k-1)}})\,:\,y\in\mathbb{F}_{q^m}\},
\end{align*}
for every $i,j \in \{1,\ldots,t\}$.
Then $(\mc{U}_1,\ldots,\mc{U}_t)$ and $(\mc{V}_1,\ldots,\mc{V}_t)$ are equivalent if and only if there exists $\xi\in\mathbb{F}_{q}^*$ such that $\{\N_s(\beta_j)\}_{j\in\{1,\ldots,t\}}=\{\xi \N_s(\alpha_i)\}_{i\in\{1,\ldots,t\}}$.
\end{theorem}
\begin{proof}
The $q$-systems $(\mc{U}_1,\ldots,\mc{U}_t)$ and $(\mc{V}_1,\ldots,\mc{V}_t)$ are equivalent if and only if 
there exist $A\in\GL(k,q^m),\,\sigma\in \Sym(t),\,(\gamma_1,\ldots,\gamma_t)\in(\mathbb{F}_{q^m}^*)^t$ such that $\mc{U}_i=\gamma_iA\mc{V}_{\sigma(i)}$ for every $i\in\{1,\ldots,t\}$, i.e. for every $x \in \fqm$ there exists $y \in \fqm$ such that
\begin{equation}
\label{eq:lineq}
    \gamma_i\begin{pmatrix}
        \textbf{a}_1\\
        \textbf{a}_2\\
        \vdots\\
        \textbf{a}_k
    \end{pmatrix}\begin{pmatrix}
        x\\
        \alpha_ix^{q^s}\\
        \vdots\\
        \alpha_i^{1+\ldots+q^{s(k-2)}}x^{q^{s(k-1)}}
    \end{pmatrix}=\begin{pmatrix}
        y\\
        \beta_{\sigma(i)}y^{q^s}\\
        \vdots\\
        \beta_{\sigma(i)}^{1+\ldots+q^{s(k-2)}}y^{q^{s(k-1)}}
    \end{pmatrix},
\end{equation}
for every $i\in\{1,\ldots,t\}$ and where $\mathbf{a}_1,\ldots,\mathbf{a}_k$ are the rows of $A$.
Using that
\begin{equation*}
    \mc{U}_i=\begin{pmatrix}
        1&0&\dots&0\\
        0&\alpha_i&&\vdots\\
        \vdots&&\ddots&0\\
        0&\dots&0&\alpha_i^{1+\ldots+q^{s(k-2)}}
    \end{pmatrix}G_{k,s,m}
\end{equation*}
and
\begin{equation*}
    \mc{V}_{\sigma(i)}=\begin{pmatrix}
        1&0&\dots&0\\
        0&\beta_{\sigma(i)}&&\vdots\\
        \vdots&&\ddots&0\\
        0&\dots&0&\beta_{\sigma(i)}^{1+\ldots+q^{s(k-2)}}
    \end{pmatrix}G_{k,s,m},
\end{equation*}
we find that \eqref{eq:lineq} is equivalent to 
\begin{equation*}
    \begin{pmatrix}
        1&0&\dots&0\\
        0&\beta_{\sigma(i)}^{-1}&&\vdots\\
        \vdots&&\ddots&0\\
        0&\dots&0&(\beta_{\sigma(i)}^{-1})^{1+\ldots+q^{s(k-2)}}
    \end{pmatrix}\gamma_iA\begin{pmatrix}
        1&0&\dots&0\\
        0&\alpha_i&&\vdots\\
        \vdots&&\ddots&0\\
        0&\dots&0&\alpha_i^{1+\ldots+q^{s(k-2)}}
    \end{pmatrix}G_{k,s,m}=G_{k,s,m}.
\end{equation*}
Hence, this is equivalent to asking
\[ \begin{pmatrix}
        1&0&\dots&0\\
        0&\beta_{\sigma(i)}^{-1}&&\vdots\\
        \vdots&&\ddots&0\\
        0&\dots&0&(\beta_{\sigma(i)}^{-1})^{1+\ldots+q^{s(k-2)}}
    \end{pmatrix}\gamma_iA\begin{pmatrix}
        1&0&\dots&0\\
        0&\alpha_i&&\vdots\\
        \vdots&&\ddots&0\\
        0&\dots&0&\alpha_i^{1+\ldots+q^{s(k-2)}}
    \end{pmatrix} \in \mathrm{stab}_{\mathrm{GL}(k,q^m)}(G_{k,s,m}), \]
that is, by Theorem \ref{th:eqgab}, there exists $d_i\in\mathbb{F}_{q^m}^*$ such that 
\begin{equation*}
    \begin{pmatrix}
        1&0&\dots&0\\
        0&\beta_{\sigma(i)}^{-1}&&\vdots\\
        \vdots&&\ddots&0\\
        0&\dots&0&(\beta_{\sigma(i)}^{-1})^{1+\ldots+q^{s(k-2)}}
    \end{pmatrix}\gamma_iA\begin{pmatrix}
        1&0&\dots&0\\
        0&\alpha_i&&\vdots\\
        \vdots&&\ddots&0\\
        0&\dots&0&\alpha_i^{1+\ldots+q^{s(k-2)}}
    \end{pmatrix}=\begin{pmatrix}
        d_i&0&\dots&0\\
        0&d_i^{q^s}&&\vdots\\
        \vdots&&\ddots&0\\
        0&\dots&0&d_i^{q^{s(k-1)}}\end{pmatrix}.
\end{equation*}
Thus, this can happen if and only if the matrix $A$ is of the form
\begin{equation*}
    A=\begin{pmatrix}
        \gamma_i^{-1}d_i&0&\dots&0\\
        0&\gamma_i^{-1}\beta_{\sigma(i)}d_i^{q^s}\alpha_i^{-1}&&\vdots\\
        \vdots&&\ddots&0\\
        0&\dots&0&\gamma_i^{-1}\beta_{\sigma(i)}^{1+\ldots+q^{s(k-2)}}d_i^{q^{s(k-1)}}{(\alpha_i^{-1})}^{1+\ldots+q^{s(k-2)}}\end{pmatrix}.
\end{equation*}
Now, suppose that the $q$-systems $(\mc{U}_1,\ldots,\mc{U}_t)$ and $(\mc{V}_1,\ldots,\mc{V}_t)$ are equivalent, then the matrix $A$ does not depend on the index $i$. Hence, we obtain, for every $i \in \{1,\ldots,t\}$, the following system
\begin{equation*}
    \begin{cases}
       \gamma_1^{-1}d_1=\gamma_i^{-1}d_i,\\
        \gamma_1^{-1}\beta_{\sigma(1)}d_1^{q^s}\alpha_1^{-1}=\gamma_i^{-1}\beta_{\sigma(i)}d_i^{q^s}\alpha_i^{-1},\\
        \qquad\qquad\qquad\vdots\\
        \gamma_1^{-1}\beta_{\sigma(1)}^{1+\ldots+q^{s(k-2)}}d_1^{q^{s(k-1)}}{(\alpha_1^{-1})}^{1+\ldots+q^{s(k-2)}}=\gamma_i^{-1}\beta_{\sigma(i)}^{1+\ldots+q^{s(k-2)}}d_i^{q^{s(k-1)}}{(\alpha_i^{-1})}^{1+\ldots+q^{s(k-2)}}.
    \end{cases}
\end{equation*}
Substituting the first equation into the second, we get, for every $i \in \{1,\ldots,t\}$,
\begin{equation*}
    \gamma_1^{-1+q^s}\gamma_i^{1-q^s}\beta_{\sigma(1)}\alpha_1^{-1}\beta_{\sigma(i)}^{-1}\alpha_i=1
\end{equation*}
and applying the norm over $\fq$ on both sides
\begin{equation*}
\mathrm{N}_{s}(\beta_{\sigma(1)}\alpha_1^{-1})=\mathrm{N}_{s}(\beta_{\sigma(i)}\alpha_i^{-1}).
\end{equation*}
Denoting $\xi:=\mathrm{N}_{s}(\beta_{\sigma(1)}\alpha_1^{-1})\in\mathbb{F}_q^*$, we have
\begin{equation*}
    \begin{cases}
        \mathrm{N}_{s}(\beta_{\sigma(1)})=\xi \mathrm{N}_{s}(\alpha_1),\\
        \mathrm{N}_{s}(\beta_{\sigma(2)})=\xi \mathrm{N}_{s}(\alpha_2),\\
        \qquad\qquad\vdots\\
        \mathrm{N}_{s}(\beta_{\sigma(t)})=\xi \mathrm{N}_{s}(\alpha_t).\\
    \end{cases}
\end{equation*}
Equivalently, for every $i\in\{1,\ldots,t\}$ there exists a $j\in\{1,\ldots,t\}$ such that $\mathrm{N}_{s}(\beta_j)=\xi \mathrm{N}_{s}(\alpha_i)$, hence the assertion is proved.\\
Now, suppose that there exists $\xi\in\F_q^*$ such that $\{\mathrm{N}_{s}(\beta_j)\}_{j\in\{1,\ldots,t\}}=\{\xi \mathrm{N}_{s}(\alpha_i)\}_{i\in\{1,\ldots,t\}}$, then there exists $\sigma\in\Sym(t)$ such that, for every $i\in\{1,\ldots,t\}$, $\sigma(i)$ is the unique $j\in\{1,\ldots,t\}$ with $\mathrm{N}_{s}(\beta_j)=\xi\mathrm{N}_{s}(\alpha_i)$. In particular, for every $i\in\{1,\ldots,t\}$, we have
\begin{equation*}
\begin{cases}
    \N_s(\beta_{\sigma(1)}\alpha_1^{-1})=\xi\\
    \N_s(\beta_{\sigma(i)}\alpha_i^{-1})=\xi   
\end{cases}\qquad\Leftrightarrow \qquad\N_s(\beta_{\sigma(1)}\alpha_1^{-1}\beta_{\sigma(i)}^{-1}\alpha_i)=1.
\end{equation*}
Therefore, for every $i\in\{1,\ldots,t\}$ there exists $\eta_i\in\F_{q^m}^*$ such that 
\begin{equation}
\label{eq:pot}
\beta_{\sigma(1)}\alpha_1^{-1}\beta_{\sigma(i)}^{-1}\alpha_i=\eta_i^{q^s-1}.
\end{equation}
Consider $\mc{U}_1$ and $\mc{V}_{\sigma(1)}$ and observe that the matrix 
\begin{equation*}A=
\begin{pmatrix}
        1&0&\dots&0\\
        0&\beta_{\sigma(1)}\alpha_1^{-1}&&\vdots\\
        \vdots&&\ddots&0\\
        0&\dots&0&\beta_{\sigma(1)}^{1+\ldots+q^{s(k-2)}}{(\alpha_1^{-1})}^{1+\ldots+q^{s(k-2)}}
        \end{pmatrix}  
\end{equation*}
maps $\mathcal{U}_1$ in $\mathcal{V}_{\sigma(1)}$.
Choosing $\gamma_1=1$ and $\gamma_i=\eta_i$ for every $i\in\{2,\ldots,t\}$ we obtain that, for every $i\in\{2,\ldots,t\}$,
\begin{equation*}
    \begin{pmatrix}
        1&0&\dots&0\\
        0&\beta_{\sigma(i)}^{-1}&&\vdots\\
        \vdots&&\ddots&0\\
        0&\dots&0&(\beta_{\sigma(i)}^{-1})^{1+\ldots+q^{s(k-2)}}
    \end{pmatrix}\eta_iA\begin{pmatrix}
        1&0&\dots&0\\
        0&\alpha_i&&\vdots\\
        \vdots&&\ddots&0\\
        0&\dots&0&\alpha_i^{1+\ldots+q^{s(k-2)}}
    \end{pmatrix}G_{k,s,m}=G_{k,s,m},
\end{equation*}
since the product matrix on the left hand side is
\begin{equation*}
    \begin{pmatrix}
        \eta_i&0&\dots&0\\
        0&\eta_i(\beta_{\sigma(i)}^{-1}\beta_{\sigma(1)}\alpha_1^{-1}\alpha_i)&&\vdots\\
        \vdots&&\ddots&0\\
        0&\dots&0&\eta_i(\beta_{\sigma(i)}^{-1}\beta_{\sigma(1)}\alpha_1^{-1}\alpha_i)^{1+\ldots+q^{s(k-2)}}
    \end{pmatrix},
\end{equation*}
and by Equation \eqref{eq:pot} we get that such a matrix corresponds to
\begin{equation*}
    \begin{pmatrix}
        \eta_i&0&\dots&0\\
        0&\eta_i^{q^s}&&\vdots\\
        \vdots&&\ddots&0\\
        0&\dots&0&\eta_i^{q^{s(k-1)}}
    \end{pmatrix},
\end{equation*}
and so it is in $\mathrm{stab}_{\mathrm{GL}(k,q^m)}(G_{k,s,m})$ and this concludes the proof.
\end{proof}

We can extend the above result to the case where $k = m$ using the same strategy. The main difference concerns the structure of $\mathrm{stab}_{\mathrm{GL}(k,q^m)}(G_{m,s,m})$, which is more difficult to handle.

\begin{theorem}
\label{th:eqsets=}
Let $s \in \{1,\ldots,m\}$ with $\gcd(s,m)=1$ and $\{\alpha_1,\ldots,\alpha_t\}$ and $\{\beta_1,\ldots,\beta_t\}$ be two subsets of $\fqm$ whose elements have pairwise distinct nonzero norms over $\fq$.
Let $(\mc{U}_1,\ldots,\mc{U}_t)$ and $(\mc{V}_1,\ldots,\mc{V}_t)$ be defined as
\begin{align*}
    \mc{U}_i&=\{(x,\mathrm{N}_s^1(\alpha_i)x^{q^s},\ldots,\mathrm{N}_s^{m-1}(\alpha_i)x^{q^{s(m-1)}})\,:\,x\in\mathbb{F}_{q^m}\},\\
    \mc{V}_j&=\{(y,\mathrm{N}_s^1(\beta_j)y^{q^{s}},\ldots,\mathrm{N}_s^{m-1}(\beta_j)y^{q^{s(m-1)}})\,:\,y\in\mathbb{F}_{q^m}\},
\end{align*}
for every $i,j \in \{1,\ldots,t\}$.
Then $(\mc{U}_1,\ldots,\mc{U}_t)$ and $(\mc{V}_1,\ldots,\mc{V}_t)$ are equivalent if and only if there exists $\xi\in\mathbb{F}_{q}^*$ such that $\{\N_s(\beta_j)\}_{j\in\{1,\ldots,t\}}=\{\xi \N_s(\alpha_i)\}_{i\in\{1,\ldots,t\}}$.
\end{theorem}
\begin{proof}
Arguing as in the previous theorem, we find that $(\mc{U}_1,\ldots,\mc{U}_t)$ and $(\mc{V}_1,\ldots,\mc{V}_t)$ are equivalent if and only if
\begin{equation*}
    \begin{pmatrix}
        1&0&\dots&0\\
        0&\beta_{\sigma(i)}^{-1}&&\vdots\\
        \vdots&&\ddots&0\\
        0&\dots&0&(\beta_{\sigma(i)}^{-1})^{1+\ldots+q^{s(m-2)}}
    \end{pmatrix}\gamma_iA\begin{pmatrix}
        1&0&\dots&0\\
        0&\alpha_i&&\vdots\\
        \vdots&&\ddots&0\\
        0&\dots&0&\alpha_i^{1+\ldots+q^{s(m-2)}}
    \end{pmatrix}G_{m,s,m}=G_{m,s,m}.
\end{equation*}
By Theorem \ref{th:eqgab=}, this happens if and only if there exists $f_i=d^{(i)}_0x+\ldots+d^{(i)}_{m-1}x^{q^{s(m-1)}}\in\tilde{\mathcal{L}}_{m,q^s}$ invertible such that
\begin{equation*}
    \begin{pmatrix}
        1&0&\dots&0\\
        0&\beta_{\sigma(i)}^{-1}&&\vdots\\
        \vdots&&\ddots&0\\
        0&\dots&0&(\beta_{\sigma(i)}^{-1})^{1+\ldots+q^{s(m-2)}}
    \end{pmatrix}\gamma_i\begin{pmatrix}
        \mathbf{a}_1\\
        \mathbf{a}_2\\
        \vdots\\
        \mathbf{a}_m
    \end{pmatrix}\begin{pmatrix}
        1&0&\dots&0\\
        0&\alpha_i&&\vdots\\
        \vdots&&\ddots&0\\
        0&\dots&0&\alpha_i^{1+\ldots+q^{s(m-2)}}
    \end{pmatrix}=D_{f_i}.
\end{equation*}

The above expression can be written as follows 

\begin{equation}\label{eq:condDicksonsecondcase}
\begin{cases}
\gamma_i\mathbf{a}_1&=(d^{(i)}_0,\alpha_i^{-1}d^{(i)}_1,\ldots,(\alpha_i^{-1})^{1+\ldots+q^{s(m-2)}}d^{(i)}_{m-1}),\\
\gamma_i\mathbf{a}_2&=\beta_{\sigma(i)}((d^{(i)}_{m-1})^{q^s},\alpha_i^{-1}(d^{(i)}_0)^{q^s},\ldots,(\alpha_i^{-1})^{1+\ldots+q^{s(m-2)}}(d^{(i)}_{m-2})^{q^s}),\\
&\;\vdots\\
\gamma_i\mathbf{a}_m&=\beta_{\sigma(i)}^{1+\ldots+q^{s(m-2)}}((d^{(i)}_{1})^{q^{s(m-1)}},\alpha_i^{-1}(d^{(i)}_2)^{q^{s(m-1)}},\ldots,(\alpha_i^{-1})^{1+\ldots+q^{s(m-2)}}(d^{(i)}_0)^{q^{s(m-1)}}).
\end{cases}    
\end{equation}
Now, suppose that $(\mc{U}_1,\ldots,\mc{U}_t)$ and $(\mc{V}_1,\ldots,\mc{V}_t)$ are equivalent and write $\mathbf{a}_1=(a_0,\ldots,a_{m-1})$. The first equation in \eqref{eq:condDicksonsecondcase} becomes
\begin{equation*}
\begin{cases}
    d_0^{(i)}=\gamma_ia_0,\\
    d_k^{(i)}=\gamma_i(\alpha_i)^{\sum_{l=0}^{k-1}q^{sl}}a_k,\qquad k\in\{1,\ldots,m-1\},
\end{cases}
\end{equation*}
thus, we can substitute these values in the second equation of \eqref{eq:condDicksonsecondcase} and obtain the following
\begin{equation*}
    \mathbf{a}_2=\beta_{\sigma(i)}\gamma_i^{q^s-1}\alpha_i^{-1}(\N_s(\alpha_i)a_{m-1}^{q^s},a_0^{q^s},\ldots,a_{m-2}^{q^s}).
\end{equation*}
In particular,
\begin{equation*}
    a_{2,2}=\beta_{\sigma(i)}\gamma_i^{q^s-1}a_{0}^{q^s}\alpha_i^{-1}.
\end{equation*}

Since the matrix $A$ does not depend on the choice of $i$, there exists $c\in\F_{q^m}^*$ such that, for every $i\in\{1,\ldots,t\}$,
\begin{equation*}
\beta_{\sigma(i)}\alpha_i^{-1}\gamma_i^{q^s-1}=c.
\end{equation*}
Hence, we have
\begin{equation*}
    \gamma_1^{-1+q^s}\gamma_i^{1-q^s}\beta_{\sigma(1)}\alpha_1^{-1}\beta_{\sigma(i)}^{-1}\alpha_i=1,
\end{equation*}
for every $i\in\{1,\ldots,t\}$.
By applying the norm on both sides, we get
\begin{equation*}
\N_s(\beta_{\sigma(1)}\alpha_1^{-1})=\N_s(\beta_{\sigma(i)}\alpha_i^{-1}).
\end{equation*}
Denoting $\xi=\N_s(\beta_{\sigma(1)}\alpha_1^{-1})\in\mathbb{F}_q^*$, we have
\begin{equation*}
    \begin{cases}
        \N_s(\beta_{\sigma(1)})=\xi \N_s(\alpha_1),\\
        \N_s(\beta_{\sigma(2)})=\xi \N_s(\alpha_2),\\
        \qquad\qquad\vdots\\
        \N_s(\beta_{\sigma(t)})=\xi \N_s(\alpha_t),\\
    \end{cases}
\end{equation*}
i.e. the assertion.

For the converse, one can argue as in Theorem \ref{th:eqsets}.
\end{proof}

In the next section, by combining the results of this section, we will give a closed formula to determine the number of inequivalent linearized Reed–Solomon codes.

\section{Counting inequivalent Linearized Reed--Solomon codes}\label{sec:6}

In this section, we determine the number of inequivalent classes of linearized Reed-Solomon codes with fixed parameters. By Theorems \ref{th:eqsets} and \ref{th:eqsets=}, two such codes are equivalent precisely when the corresponding evaluation sets differ by multiplication by an element of $\F_q^*$. Consequently, the problem reduces to counting subsets of $\F_q^*$ of a given size up to this multiplicative action.

More precisely, let $\binom{\F_q^*}{t}$ denote the family of subsets of $\F_q^*$ of cardinality $t$. The multiplicative group $\F_q^*$ acts naturally on $\binom{\F_q^*}{t}$ by scalar multiplication. Therefore, inequivalent linearized Reed-Solomon codes correspond to the orbits of this action. Our goal is thus to determine the number of such orbits.

To this end, we study the structure of the stabilizers of subsets under this action. Since $\F_q^*$ is cyclic of order $q-1$, its subgroups are uniquely determined by the divisors of $q-1$, and this property allows us to explicitly characterize the possible stabilizers. Using this characterization together with the Orbit-Stabilizer Theorem, we derive a counting formula for the number of orbits, and hence for the number of inequivalent classes of linearized Reed--Solomon codes.

Consider the following action

\begin{align*}
    \Phi_\star:\quad\mathbb{F}_q^*\times\binom{\F_q^*}{t}&\rightarrow \quad\binom{\F_q^*}{t}\:,\\
    \\
    (\xi,A)\quad&\mapsto \quad \xi \star A
\end{align*}
where for every $A=\{a_i\}_{i\in\{1,\ldots,t\}}\in\binom{\F_q^*}{t}$ and $\xi\in\mathbb{F}_q^*$, we define $\xi \star A\coloneqq \{\xi a_i\}_{i\in\{1,\ldots,t\}}$.

\begin{remark}\label{rem:cyclicsubgroups}
Observe that $\F_q^*$ is cyclic of order $q-1$, therefore, for every divisor 
$d \mid (q-1)$ there exists a unique subgroup $\cyclic_d$ of order $d$. 
Therefore, the stabilizer of any $A \in \binom{\F_q^*}{t}$ must be one of 
these subgroups. In particular, by the Orbit--Stabilizer Theorem, the size 
of every orbit divides $q-1$.
\end{remark}

In order to count the orbits of the action $\Phi_\star$, we first characterize the subsets whose stabilizer contains a given subgroup of $\F_q^*$. The following lemma shows that this occurs precisely when the subset can be written as a union of multiplicative cosets.

\begin{lemma}
\label{lem:char}
    Let $A\in\binom{\F_q^*}{t}$ and $d\,|\,q-1$, then $\cyclic_d\leq\stab(A)$ if and only if $d\,|\,t$ and there exist $a_1,\ldots,a_{t/d}\in A$ such that 
    $A=a_1\cyclic_d\cup\ldots\cup a_{t/d}\cyclic_d$.
\end{lemma}
\begin{proof}
Suppose that $\cyclic_d\leq\stab(A)$, since $\cyclic_d$ is a cyclic subgroup, there exist $\xi\in\F_q^*$ such that $\cyclic_d=\la\xi\ra$ and $A=\xi^j \star A$ for every $j\in\{0\ldots,d-1\}$. Consider $a_1\in A$, then $a_1\cyclic_d=\{a_1,a_1\xi,\ldots,a_1\xi^{d-1}\}\subseteq A$. Similarly, we can now consider $a_2\in A\setminus a_1\cyclic_d$  to obtain another multiplicative coset $a_2\cyclic_d$. Repeating this process until $a_i$ cannot be chosen, we end up with a partition of $A=a_1\cyclic_d\cup\ldots\cup a_{t/d}\cyclic_d$. Here, clearly $d\,|\,t$.

Now, suppose that $d\,|\,t$ and that there exist $a_1,\ldots,a_{t/d}\in A$ such that 
 $A=a_1\cyclic_d\cup\ldots\cup a_{t/d}\cyclic_d$. Then, for every $a\in A$ there exist some $a_i\in A$ and $j\in\{0,\ldots,d-1\}$ such that $a=a_i\xi^j$, so $\xi^k a=a_i\xi^{j+k}\in a_i\cyclic_d\subseteq A$ for every $k\in\{0,\ldots,d-1\}$, hence $\cyclic_d\leq \stab(A)$.
\end{proof}

\begin{remark}
From the previous lemma we directly obtain that the stabilizers must be chosen between the subgroups of $\F_q^*$ whose order divides $\gcd(q-1,t)$. 
\end{remark}

Lemma \ref{lem:char} directly implies the following.

\begin{theorem}\label{thm:countingmtcoprime}
Let $q,t,m,k$ be positive integers such that $\gcd(q-1,t)=1$ and $1<k\leq m$, then the number of inequivalent classes of $[\mathbf{m},k,tm-k+1]_{q^m/q}$ linearized Reed-Solomon codes equals $\Psi(k,m)\binom{q-1}{t}/(q-1)$, where 
\[ \Psi(k,m)= \begin{cases}
\frac{\varphi(m)}2, & \text{ if } k\leq m-2,\\
1, & \text{ if } k \in \{m-1,m\},\\
\end{cases}\]
and $\varphi$ is the Euler's phi function.
\end{theorem}
\begin{proof}
Since $\gcd(q-1,t)=1$, it follows for every $A\in\binom{\F_q^*}{t}$ with $|O_{\F_q^*}(A)|=q-1$, that the stabilizer $\mathrm{stab}(A)=\{1\}$. Therefore, in this particular case, every set gives rise to a class of cardinality $q-1$, thus the number of inequivalent classes is $\binom{q-1}{t}/(q-1)$.
By Proposition \ref{prop:s<m/2}, different choices of $1 \leq s< m/2$ with $\gcd(s,m)=1$ provide inequivalent codes when $1<k<m-1$ and so we get the assertion.
\end{proof}

To determine the number of inequivalent classes in the general case, we study the orbits of the action $\Phi_\star$. By Lemma~\ref{lem:char}, subsets can be characterized by their stabilizers, which correspond precisely to representations as unions of multiplicative cosets.

\subsection{A closed formula}

Consider $q,t,m,k$ be positive integers with $k\leq m$.
Let us start by fixing $d$ as a divisor of $\gcd(q-1,t)$ and define
\begin{equation*}
    f_d=\bigg|\bigg\{A\in\binom{\F_q^*}{t}\,:\,|O_{\F_q^*}(A)|=\frac{q-1}{d}\bigg\}\bigg|
\end{equation*}
then, the number of orbits of size $(q-1)/d$ is given by $\frac{f_d}{(q-1)/d}$. By the Orbit-Stabilizer theorem and by Remark \ref{rem:cyclicsubgroups}, the value $f_d$ can be written as follows
\begin{align*}
    f_d&=\bigg|\bigg\{A\in\binom{\F_q^*}{t}\,:\,\stab(A)=\cyclic_d\bigg\}\bigg|\\
    &=\bigg|\bigg\{A\in\binom{\F_q^*}{t}\,:\,\stab(A)\geq \cyclic_d\bigg\}\bigg|-\bigg|\bigg\{A\in\binom{\F_q^*}{t}\,:\,\stab(A)>\cyclic_d\bigg\}\bigg|.
\end{align*}
By Lemma \ref{lem:char} we know that 
\begin{align*}
    \bigg|\bigg\{A\in\binom{\F_q^*}{t}\,:\,\stab(A)\geq \cyclic_d\bigg\}\bigg|=\bigg|\bigg\{A\in\binom{\F_q^*}{t}\,:\,A=a_1\cyclic_d\cup\ldots\cup a_{t/d}\cyclic_d,\,\exists a_1,\ldots,a_{t/d}\in A\bigg\}\bigg|,
\end{align*}
thus
\begin{align*}
    f_d=\binom{\frac{q-1}{d}}{\frac{t}{d}}-\bigg|\bigg\{A\in\binom{\F_q^*}{t}\,:\,\stab(A)>\cyclic_d\bigg\}\bigg|.
\end{align*}
On the other hand, we can write
\begin{align*}
    \bigg\{A\in\binom{\F_q^*}{t}\,:\,\stab(A)>\cyclic_d\bigg\}=\bigcup_{\substack{d'\,s.t.\, d'\neq d \\ d\,|\,d'\,|\,(q-1,t)}}\bigg\{A\in\binom{\F_q^*}{t}\,:\,\stab(A)=\cyclic_{d'}\bigg\},
\end{align*}
where the union is clearly disjoint. Finally, we obtain

\begin{align*}
    f_d&=\binom{\frac{q-1}{d}}{\frac{t}{d}}-\sum_{\substack{d'\,s.t.\, d'\neq d \\ d\,|\,d'\,|\,(q-1,t)}}\bigg|\bigg\{A\in\binom{\F_q^*}{t}\,:\,\stab(A)=\cyclic_{d'}\bigg\}\bigg|\\
    &=\binom{\frac{q-1}{d}}{\frac{t}{d}}-\sum_{\substack{d'\,s.t.\, d'\neq d \\ d\,|\,d'\,|\,(q-1,t)}}f_{d'}.
\end{align*}

In conclusion, the total number of orbits is
\begin{equation*}
    \sum_{d\,|\,(q-1,t)}\frac{f_d}{(q-1)/d}=\sum_{d\,|\,(q-1,t)}\frac{\binom{\frac{q-1}{d}}{\frac{t}{d}}-\sum_{\substack{d'\,s.t.\, d'\neq d \\ d\,|\,d'\,|\,(q-1,t)}}f_{d'}}{(q-1)/d}.
\end{equation*}

As in the proof of Theorem \ref{thm:countingmtcoprime}, Proposition \ref{prop:s<m/2} guaranties that different choices of $1 \leq s< m/2$ with $\gcd(s,m)=1$ provide inequivalent codes.
All of the above considerations yield the following result.

\begin{theorem}
    Suppose that $1<k\leq m$, then the number of inequivalent classes for linearized Reed-Solomon codes with parameters $[\mathbf{m},k,tm-k+1]_{q^m/q}$ equals
    $$\Psi(k,m)\sum_{d\,|\,(q-1,t)}\frac{f_d}{(q-1)/d}\\
    =\Psi(k,m)\sum_{d\,|\,(q-1,t)}\frac{\binom{\frac{q-1}{d}}{\frac{t}{d}}-\sum_{\substack{d'\,s.t.\, d'\neq d \\ d\,|\,d'\,|\,(q-1,t)}}f_{d'}}{(q-1)/d}.$$
\end{theorem}

As a special case, when $t=1$, we recover Corollary \ref{cor:numbinequiGabidulin} for Gabidulin codes.

\subsection{Examples}

In this section, we present two examples that illustrate how to use the previous argument to count the number of inequivalent linearized Reed–Solomon codes of a given dimension. We begin by proving the following result.

\begin{proposition}
The number of inequivalent classes of 
linearized Reed--Solomon codes with parameters $[\mathbf{m},k]_{25^m/25}$, with $\mathbf{m}=( \underbrace{m,\ldots,m}_{12\text{ times}} )$ and $1<k\leq m$,
equals $112720$ if $k\in\{m-1,m\}$ and it is $\frac{\varphi(m)}2 112720$ otherwise.
\end{proposition}
\begin{proof}
Consider $q=25$ and $t=12$. We study the action of $\F_{25}^*$ on 
$\binom{\F_{25}^*}{12}$. Since $\F_{25}^*$ is cyclic of order $24$, we may 
identify it with the cyclic group $\cyclic_{24}$. As $\gcd(24,12)=12$, only the 
subgroups whose order divides $12$ may occur as stabilizers. The relevant 
portion of the subgroup lattice of $\cyclic_{24}$ is illustrated in Figure~\ref{fig:lattice}.

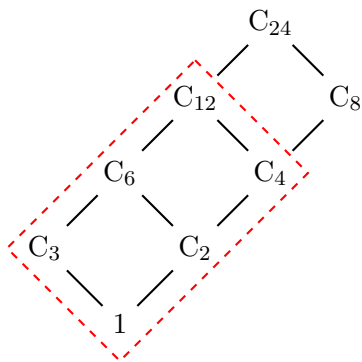
\begin{figure}[h]
\centering
\begin{tikzpicture}

\node (24) at (0,0) {C$_{24}$};
\node (12) at (-1,-1) {C$_{12}$};
\node (8) at (1,-1) {C$_8$};
\node (6) at (-2,-2) {C$_{6}$};
\node (4) at (0,-2) {C$_4$};
\node (3) at (-3,-3) {C$_{3}$};
\node (2) at (-1,-3) {C$_2$};
\node (1) at (-2,-4) {$1$};
\draw[black, thick] (24) to (12);
\draw[black, thick] (24) to (8);
\draw[black, thick] (12) to (4);
\draw[black, thick] (12) to (6);
\draw[black, thick] (8) to (4);
\draw[black, thick] (6) to (3);
\draw[black, thick] (6) to (2);
\draw[black, thick] (4) to (2);
\draw[black, thick] (3) to (1);
\draw[black, thick] (2) to (1);
\draw[red,thick, dashed](-1,-0.5)--(0.5,-2)--(-2,-4.5)--(-3.5,-3)--(-1,-0.5);
\end{tikzpicture}
\caption{Subgroup lattice of the cyclic group $C_{24}$. The highlighted portion corresponds to subgroups whose order divides $\gcd(24,12)=12$.}
\label{fig:lattice}
\end{figure}

Let 
\[
D=\{1,2,3,4,6,12\}
\]
be the set of divisors of $12$. For each $d\in D$, we compute the number 
\[
f_d=\big|\{A\in\binom{\F_{25}^*}{12}:\stab(A)=\cyclic_d\}\big|
\]
and the corresponding number of orbits 
\[
O_d=\frac{f_d}{24/d}.
\]

We begin with the largest stabilizer.\\

\textbf{Stabilizer $\cyclic_{12}$.}\\
Since $12=\gcd(24,12)$, no larger subgroup must be excluded. Hence
\[
f_{12}=\binom{24/12}{12/12}=\binom{2}{1}=2.
\]
Each orbit has size $24/12=2$, so
\[
O_{12}=\frac{f_{12}}{2}=1.
\]

\textbf{Stabilizer $\cyclic_{6}$.}\\
Here we must subtract the sets already counted for $\cyclic_{12}$:
\[
f_{6}=\binom{24/6}{12/6}-f_{12}
      =\binom{4}{2}-2
      =4.
\]
Since each orbit has size $24/6=4$,
\[
O_6=\frac{f_6}{4}=1.
\]

\textbf{Stabilizer $\cyclic_{4}$.}\\
Again removing the sets with stabilizer $\cyclic_{12}$ gives
\[
f_4=\binom{24/4}{12/4}-f_{12}
    =\binom{6}{3}-2
    =18.
\]
The orbit size is $24/4=6$, hence
\[
O_4=\frac{f_4}{6}=3.
\]

\textbf{Stabilizer $\cyclic_{3}$.}\\
Since $\cyclic_3$ is contained in both $\cyclic_6$ and $\cyclic_{12}$,
we subtract the corresponding contributions:
\[
f_3=\binom{24/3}{12/3}-f_6-f_{12}
    =\binom{8}{4}-4-2
    =64.
\]
The orbit size is $24/3=8$, yielding
\[
O_3=\frac{f_3}{8}=8.
\]

\textbf{Stabilizer $\cyclic_{2}$.}\\
Similarly,
\[
f_2=\binom{24/2}{12/2}-f_4-f_6-f_{12}
    =\binom{12}{6}-18-4-2
    =900,
\]
and therefore
\[
O_2=\frac{f_2}{12}=75.
\]

\textbf{Trivial stabilizer.}\\
Finally,
\[
f_1=\binom{24}{12}-f_2-f_3-f_4-f_6-f_{12}
    =\binom{24}{12}-900-64-18-4-2
    =2703168,
\]
so
\[
O_1=\frac{f_1}{24}=112632.
\]

Summing all orbit counts, we obtain
\[
\#\text{Orbits}
=O_{12}+O_6+O_4+O_3+O_2+O_1
=112720.
\]
\end{proof}

Finally, we conclude with an example illustrating how to count the number of inequivalent Linearized Reed–Solomon codes and determine the corresponding equivalence classes for smaller parameter choices.

\begin{proposition}
The number of inequivalent classes of 
linearized Reed-Solomon codes with parameters $[\mathbf{m},k]_{7^m/7}$, with $1<k\leq m$, $\mathbf{m}=(m,m,m)$,
equals $4$ if $k\in\{m-1,m\}$ otherwise is $2\varphi(m)$. Identifying $\F_7$ with $\Z_7$, the equivalence classes of codes are represented by $LR_k^{q^s}[(\alpha_1,\alpha_2,\alpha_3)]$ with $\{ \alpha_1,\alpha_2,\alpha_3\}$ are as in Table \ref{table:finexample}.
\begin{table}[h]
\centering
\begin{tabular}{|c|c|}
\hline
$\{\alpha_1,\alpha_2,\alpha_3\}$ & $O(\{\alpha_1,\alpha_2,\alpha_3\})$ \\
\hline
$\{1,2,4\}$ & $\{\{1,2,4\},\{3,5,6\}\}$ \\
\hline
$\{1,5,6\}$ & $\{\{1,5,6\},\{3,4,6\},\{2,4,5\},\{1,2,6\},\{2,3,5\},\{1,3,4\}\}$ \\
$\{1,3,5\}$ & $\{\{1,3,5\},\{2,3,6\},\{1,4,5\},\{1,2,3\},\{2,4,6\},\{4,5,6\}\}$ \\
$\{2,5,6\}$ & $\{\{2,5,6\},\{1,4,6\},\{2,3,4\},\{1,3,6\},\{1,2,5\},\{3,4,5\}\}$ \\
\hline
\end{tabular}
\caption{Orbits of the action of $\Z_7^*$ on $3$-subsets $\Z_7^*$.}
\label{table:finexample}
\end{table}
\end{proposition}
\begin{proof}
 Consider the action of $\Z_7^*$ on $\binom{\Z_7^*}{3}$ through $\Phi_\star$. The structure of $\Z_7^*$ is described in Figure \ref{fig:2}.

\begin{figure}[H]
\centering
\begin{tikzpicture}
\node (6) at (-2,-2) {C$_{6}$};
\node (3) at (-3,-3) {C$_{3}$};
\node (2) at (-1,-3) {C$_2$};
\node (1) at (-2,-4) {$1$};

\draw[black, thick] (6) to (3);
\draw[black, thick] (6) to (2);
\draw[black, thick] (3) to (1);
\draw[black, thick] (2) to (1);
\draw[red,thick, dashed](-3,-2.5)--(-1.5,-4)--(-2,-4.5)--(-3.5,-3)--(-3,-2.5);
\end{tikzpicture}
\caption{Subgroup lattice of the cyclic group $C_{6}$. The highlighted portion corresponds to subgroups whose order divides $\gcd(6,3)=3$.}
\label{fig:2}
\end{figure}
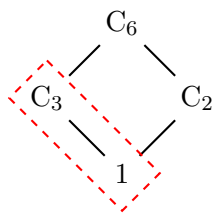

\textbf{Stabilizer $\cyclic_{3}$.}\\
Since we consider $\Z_7^*$ as the acting group, we can write $\cyclic_3=\langle2\rangle=\{1,2,4\}$. The orbits cardinalities can be $2$ or $6$. Firstly, we count the sets whose orbit has size $2$, the stabilizer of these elements is the subgroup $\cyclic_3$. By Lemma \ref{lem:char} these elements can be written as unions of cosets, in this case they are exactly the elements of $\cyclic_6/\cyclic_3=\{\cyclic_3,3\cyclic_3\}=\{\{1,2,4\},\{3,5,6\}\}$. Thus, we can write
\begin{equation*}
    O(\{1,2,4\})=\{\{1,2,4\},\{3,5,6\}\},
\end{equation*}
and this is the only orbit of cardinality $2$.

\textbf{Trivial stabilizer.}\\
In order to count the orbits of cardinality $6$, using the same process we consider $\cyclic_6/\{1\}=\{1,2,3,4,5,6\}$ and group these elements into sets of cardinality $3$ taking into account that we already considered the sets $\{1,2,4\}$ and $\{3,5,6\}$. Moreover, since each of these orbits has cardinality $6$ and we are left with $18$ elements, we know that we should obtain $3$ orbits. We start from $\{1,5,6\}$ and multiplying by every element of $\Z_7^*$ we obtain its orbit, that is
\begin{equation*}
O(\{ 1, 5, 6 \})=\{\{ 1, 5, 6 \},\{ 3, 4, 6 \},\{ 2, 4, 5 \},\{ 1, 2, 6 \},\{ 2, 3, 5 \},\{ 1, 3, 4 \}\}.    
\end{equation*}
Then we move to another element not already taken, for example, we can consider $\{ 1, 3, 5 \}$ and thus obtain the orbit 
\begin{equation*}
O(\{ 1, 3, 5 \})=\{\{ 1, 3, 5 \},\{ 2, 3, 6 \},\{ 1, 4, 5 \},\{ 1, 2, 3 \},\{ 2, 4, 6 \},\{ 4, 5, 6 \}\}.    
\end{equation*}
Finally, we are left with $6$ elements that belong to the same orbit, namely 

\begin{equation*}
    O(\{ 2, 5, 6 \})=\{\{ 2, 5, 6 \},\{ 1, 4, 6 \},\{ 2, 3, 4 \},\{ 1, 3, 6 \},\{ 1, 2, 5 \},\{ 3, 4, 5 \}\}.
\end{equation*}

We have divided into $4$ orbits a total of $20$ elements, since $|\binom{\Z_7^*}{3}|=20$, this ends the proof.
\end{proof}

\section{Conclusions}\label{sec:7}

In this paper, we investigated the equivalence problem for linearized Reed-Solomon codes and determined the number of inequivalent codes within this family. 
Using the geometric interpretation of sum-rank metric codes via systems of $\F_q$-subspaces, we analyzed the stabilizer of the Gabidulin system and obtained a characterization of equivalence for linearized Reed-Solomon codes in terms of the norms of their defining parameters. 
This description allowed us to reduce the classification problem to the action of $\F_q^\ast$ on subsets of $\F_q^\ast$, leading to explicit formulas for the number of inequivalent linearized Reed-Solomon codes. 
In this way, our results in the square case extend the work of Schmidt and Zhou \cite{schmidt2018number}, who determined the number of inequivalent Gabidulin codes.

Several questions remain open. 
First, throughout this paper, we focused on the square case, that is, when all block lengths are equal to $m$. 
A natural direction for further research is to study the equivalence problem for linearized Reed-Solomon codes in the more general setting where the block lengths $n_1,\ldots,n_t$ are not necessarily equal to $m$. 
Such a generalization would further extend the results of Schmidt and Zhou and could reveal a richer structure of equivalence classes in the sum-rank setting.

Another challenging problem concerns linearized Reed-Solomon codes of dimension greater than $m$. 
In the regime $k \le m$ the connection with Gabidulin systems plays a crucial role in our analysis, allowing us to exploit known structural results. 
However, when $k > m$ this connection becomes much weaker, and the techniques developed here no longer apply directly. 
Understanding the equivalence classes of linearized Reed-Solomon codes in this setting, therefore, appears to require new ideas.

\section*{Acknowledgments}
Ferdinando Zullo is very grateful for the hospitality of the Department of Mathematics and Data Science, Vrije Universiteit Brussel, Brussel, Belgium, where he spent two weeks as a visiting researcher.

The research of Jonathan Mannaert is supported by VUB OZR mandate ``An algebraic approach to Boolean functions with a geometric domain: the everlasting friendship between algebra and combinatorics’’ (OZR4414).

The research of Marta Messia is supported by FWO research project ``Finite Geometry Applications to Coding Theory and Cryptography’’ (FWO grant nr. G0AJI25N).

The research of Ferdinando Zullo was partially supported by the Italian National Group for Algebraic and Geometric Structures and their Applications (GNSAGA - INdAM).
This work was partially written during the Opera 2026 conference in Bordeaux, and so we acknowledge the support from the International Research Laboratory LYSM in partnership between CNRS and INdAM.

\bibliographystyle{abbrv}
\bibliography{biblio}
\end{document}